\documentclass[aap]{imsart}
\usepackage{latexsym,amssymb}
\usepackage{xcolor}
\usepackage{tensor}

\usepackage[pdftex]{graphicx}
\usepackage{amsmath,amsfonts}
\usepackage{multicol}
\usepackage{pifont}
\usepackage{geometry}
\usepackage{tikz}
\usepackage{amsmath,amsthm,amsfonts,amssymb,bbm}
\usepackage{graphicx,psfrag,subfigure,float,color}
\usepackage{cite}
\usepackage{graphicx}
\usetikzlibrary{arrows}
\usepgflibrary{snakes}
\usetikzlibrary{arrows,snakes,backgrounds}
\DeclareGraphicsExtensions{.jpg,.pdf,.pdftex,.eps}
\RequirePackage[colorlinks,citecolor=blue,urlcolor=blue]{hyperref}

\startlocaldefs

\newcommand{\Pb}{\mathbb{P}}

\newcommand{\dx}{\mathrm{d}}
\newcommand{\R}{\mathbb{R}}

\newcommand{\Z}{\mathbb{Z}}
\newcommand{\C}{\mathbb{C}}

\newcommand{\GUE}{\mathrm{GUE}}

\newtheorem{tthm}{Theorem}

\newtheorem{prop}{Proposition}[section]

\newtheorem{lem}[prop]{Lemma}
\newtheorem{defin}[prop]{Definition}
\newtheorem{cor}{Corollary}

\newtheorem{rem}[prop]{Remark}

\endlocaldefs

\begin{document}

\begin{frontmatter}
\title{$\mathbf{\GUE\times\GUE}$ limit law  at hard shocks in ASEP}
\runtitle{ $\mathbf{\GUE\times\GUE}$ in ASEP}

\begin{aug}
  \author{\fnms{Peter}  \snm{Nejjar}\corref{}\thanksref{t2}\ead[label=e1]{nejjar@iam.uni-bonn.de}},

  \thankstext{t2}{This article was written and submitted while the author was affiliated with IST Austria and his research was  supported by ERC Advanced Grant No. 338804 and ERC Starting Grant No. 716117. 
  Revision and acceptance of the article took place while the author was affiliated with Bonn University, where his  research was supported by  the Deutsche Forschungsgemeinschaft (German Research
Foundation) by the CRC 1060 (Projektnummer 211504053) and Germany’s Excellence Strategy - GZ 2047/1,
Projekt ID 390685813. }

  \runauthor{P. Nejjar}

  \affiliation{Institute of applied mathematics, Bonn University }

  \address{Bonn University, Endenicher Allee 60, 53115 Germany\\ 
          \printead{e1}}

\end{aug}

\begin{abstract}
We consider the asymmetric simple exclusion process (ASEP) on $\Z$ with  initial data such that in the large time particle density $\rho(\cdot)$ a discontinuity (shock)  at the origin is created. At the shock,  the value of $\rho$ jumps from zero to one, but 
$\rho(-\varepsilon),1-\rho(\varepsilon) >0 $ for any $\varepsilon>0$. We are interested in the rescaled position of a tagged particle which enters  the shock with positive probability. We show that, inside the shock region, the particle position  has the KPZ-typical $1/3$ fluctuations, 
a $F_{\GUE}\times F_{\GUE}$ limit law and a degenerated correlation length.   Outside the shock region, the particle fluctuates as if there was no shock. Our arguments are  mostly probabilistic,  in particular,   the mixing times of countable state space ASEPs  are instrumental to study  the fluctuations at shocks. 
\end{abstract}
\end{frontmatter}



\section{Introduction}
We consider the asymmetric simple exclusion process (ASEP) on $\Z$. In this model, particles move in $\Z$  and there is at most one particle per site. Each particle waits independently of all other particles  an exponential time (with parameter $1$) to attempt to move one unit step, which is a step   to the right with probability $p>1/2$, and a step 
to the left with probability $q=1-p$. The attempted jump is successful iff the target site is empty (exclusion constraint). 
ASEP is a continuous-time Markov process  with state space $X=\{0,1\}^{\Z}$ and we denote by $\eta_{\ell}\in X$ the particle configuration at time $\ell$; see \cite{Li85b} for the rigorous construction of ASEP.
If $p=1$ we speak of the totally ASEP (TASEP). 

Given an  initial data $\eta_{0}$ we can assign a label (an integer) to each particle, and we denote by $x_{M}(t)$ the position at time $t$ of the particle with label $M$.
The particle position $x_{M}(t)$ is  directly related to the  height function associated to the ASEP dynamics.  
As  growth model, ASEP belong to the Kardar-Parisi-Zhang (KPZ) universality class, see \cite{Cor11} for a review. 
The members of this class are believed to share, within a few subclasses,  a common large time fluctuation behavior. In particular, ASEP is expected to have, modulo some special situations, the same large time  fluctuation behavior for all $p\in (1/2,1].$ Since TASEP is more tractable than ASEP (due to TASEP's  imminent determinantal structure, but also because certain  probabilistic techniques such as couplings with last passage percolation can be used), many asymptotic results were first obtained for TASEP. 
 In light of the idea of universality, it is of great interest to generalize results from TASEP to the general ASEP.  Another key motivation to study ASEP is that, by considering a weakly asymmetric scaling, ASEP provides a bridge to the famous \textit{KPZ equation}.
The (Cole-Hopf) solution of the KPZ equation is the logarithm of  the solution of the stochastic heat equation with multiplicative noise, and the latter can, for certain initial data, be obtained from ASEP under weak asymmetry \cite{ACQ11}, \cite{BG97}.

The hydrodynamical behavior of ASEP is well established: For ASEP with a sequence of initial configurations  $\eta_{0}^{N}\in X,N\geq 1$, assume that
 \begin{equation}  \label{partdens}
\lim_{N \to \infty}\frac{1}{N}\sum_{i \in \Z}\delta_{\frac{i}{N}}\eta_{0}^{N}(i)=\rho_{0}(\xi)\dx \xi,
\end{equation}
where $\delta_{i/N}$ is the dirac measure at $i/N$ and the convergence is in the sense of vague convergence of measures.
Then the large time density of the ASEP is given by 
 \begin{equation}  \label{partdens2}
\lim_{N \to \infty}\frac{1}{N}\sum_{i \in \Z}\delta_{\frac{i}{N}}\eta_{\tau N}^{N}(i)=\rho(\xi,\tau)\dx \xi,
\end{equation}
where $\rho(\xi,\tau)$ is the unique entropy solution of the Burgers equation with initial data $\rho_{0}$.

A very important   result on KPZ fluctuations  for the general ASEP was obtained in \cite{TW08b}, where the authors consider ASEP with step initial data $\eta^{\mathrm{step}}=\mathbf{1}_{\Z_{\leq 0}}$ (in fact, they consider the initial data $\eta=\mathbf{1}_{\Z_{\geq 0}}$ with particles having a drift to the left, which is equivalent). The limiting particle density \eqref{partdens2} then has a region of decreasing density (rarefaction fan), and  for a particle located in this region, the fluctuations around its macroscopic position are of order $t^{1/3}$ and given by the Tracy-Widom $F_{\GUE}$ distribution  \cite[Theorem 3]{TW08b}.  For TASEP, this result had been shown earlier in \cite[Theorem 1.6]{Jo00b}. The authors of  \cite{TW08b} also obtained the limit law of the rescaled position of the particle initially at position $-M$ ($M$ fixed), see Theorem \ref{convthm} below.
The results of  \cite{TW08b} were later extended to so-called (generalized) step Bernoulli initial data \cite{TW09b}, \cite{AB16b}. For stationary ASEP (where $\eta_{0}(i),i\in \Z$ are i.i.d. Bernoulli), \cite{A18} showed that the current fluctuations along the characteristics converge to the Baik-Rains distribution, again generalizing a result known for TASEP \cite{FS06v} to the general  ASEP. Considerable effort has also been devoted to (half-) flat initial data \cite{OQR16}, \cite{OQR18}, 
which  again are already understood for TASEP \cite{BFS07}, \cite{BFPS06}.

In this paper, we consider the general  ASEP with  a shock (discontinuity) in the particle density, and our main contribution is to show that KPZ fluctuations arise at this shock. In the case of \textit{random} (independent Bernoulli) initial data, 
 shocks in the general ASEP   have been extensively studied, see \cite[Chapter 3]{Li99}  for a review. However, for such random initial data,
  the initial randomness supersedes the fluctuations of ASEP, leading to a gaussian limit law under
 $t^{1/2}$ scaling, i.e. one does not obtain KPZ fluctuations.
  We thus  consider deterministic initial data  defined in  \eqref{IC},\eqref{IC2} below. Their   macroscopic particle density is depicted in Figure \ref{curvdens}: At the origin, two rarefaction fans come together,  and $\rho(\xi,1)$ makes a jump from $0$ to $1$.
 We call this discontinuity in $\rho$ a hard shock and are  interested in the fluctuations of particles around the macroscopic shock position. We study this question for two different initial data: For \eqref{IC}, the shock region remains discrete, 
 whereas for \eqref{IC2} it grows as a power of  $t$. Our main results  - Theorem \ref{GUEGUE} and \ref{GUEGUE2} - show that inside the shock, we have the KPZ-typical $1/3$ fluctuation exponent, a degenerated $1/3$ correlation exponent, as well as a limit law given by a product of two Tracy-Widom $\GUE$ distributions. Such KPZ fluctuation behavior has previously  been only observed at "non-hard" shocks   in the totally asymmetric case (see \cite{FN14} and  \eqref{tildex} below).  
 Theorem \ref{GUEGUE} and \ref{GUEGUE2} thus give the first example of   KPZ fluctuations at shocks in the general asymmetric case.   Specializing our results to TASEP, we  show in  Corollary \ref{TASEPcor} that we can smoothly transit 
 between the hard and non-hard shock fluctuations in TASEP.  We obtained the fluctuations of the second class particle at the shock  in \cite{N20} after the present work was finished.
 
The results of \cite{FN14} (and of the subsequent works \cite{FN15},\cite{Nej18},\cite{FN17},\cite{FGN18})  were all obtained by working in a last passage percolation model, which then is coupled to TASEP. Such a coupling  does not exist for ASEP and in this paper, 
we work directly in the exclusion process.  We  thus  give  a direct understanding of the shock fluctuations without passing through an auxiliary model.  
A key difficulty in the general asymmetric case is then to provide lower bounds for particle positions  (see Proposition \ref{sepprop}). 
In this paper, we  control particle positions by comparison with countable state space ASEPs, and we control the latter by making use of their well-studied (see  \cite{BBHM},\cite{LL19})  mixing behavior.
Despite their importance and popularity, mixing times of finite/countable state space ASEPs do not seem to be a commonly used tool to show KPZ fluctuations. We can formulate the use of mixing times in terms of a general strategy;
we explain this strategy and our methodology in
 Section \ref{method}.  Let us state our  main results now. 

We first consider the case where the shock region remains discrete  and the KPZ fluctuations only appear in a double limit.  We will consider  for  $C\in\R$ the initial data 
\begin{equation}\label{IC}
x_{n}(0)=
\begin{cases}
-n-\lfloor (p-q)(t- C t^{1/2})\rfloor   \quad &\mathrm{for} \,  n \geq 1 \\
-n\quad &\mathrm{for}\, -\lfloor (p-q)(t  - C t^{1/2})\rfloor  \leq n \leq 0,
\end{cases}
\end{equation}
 and we denote by $(\eta_{\ell})_{\ell\geq 0}$ the ASEP started from this initial data.  
In order to emphasize the dependence on $t,C$ we may write $(\eta_{\ell})_{\ell\geq 0}= (\eta_{\ell}^{C,t})_{\ell\geq 0}$. We will consider $ (\eta_{\ell}^{C,t},0\leq \ell\leq t)$ and let $t $ go to infinity. Then, with the notation of \eqref{partdens},
we have $N=t$ so that $\eta^{C,t}_{0}=\eta_{0}^{N}.$  In the following, we will however suppress the $t,C$ from our notation and simply write  $(\eta_{\ell})_{\ell\geq 0}$.
The initial particle density $\rho_{0}$ from  \eqref{partdens} is depicted in Figure \ref{curvdens} on the left: The initial data \eqref{IC} has two regions where the density is one, namely an  infinite region to the left of $q-p,$ and a finite region in between $0$ and $p-q$. At time $t$, the particles from the infinite region have formed a rarefaction fan
 that \textit{ends} at the origin. Equally at time $t$, the particles from the finite region 
 have formed a rarefaction fan which \textit{begins} at the origin. 
 So at time $t,$ the two fans come together and create a shock at the origin, see
 Figure \ref{curvdens} on the right for an illustration of the large time density $\rho(\xi,1).$

We will study the position of particle $x_{M}(t), M\geq 1,$ in the double limit $M\to\infty \, t\to \infty.$
 A particular choice of the value $C$ is then 
\begin{equation}\label{C}
C=C(M)=2\sqrt{\frac{M}{p-q}}, \quad M\in \Z_{\geq1}.
\end{equation}
  The scaling \eqref{C} is precisely so that it is particle $x_{M}$ which at time $t$ is located around the origin: If all the particles $x_{n},n\leq 0,$ 
 were absent from the system, then $x_{M}(t)$ would have $M^{-1/6}t^{1/2}$ fluctuations  around the origin and converge to a single $F_{\GUE}$ (see below)  distribution as $M\to\infty \, t\to \infty$. Because of the shock though,   $x_{M}(t)$ has very different fluctuation behavior at the origin, see Theorem \ref{GUEGUE}.


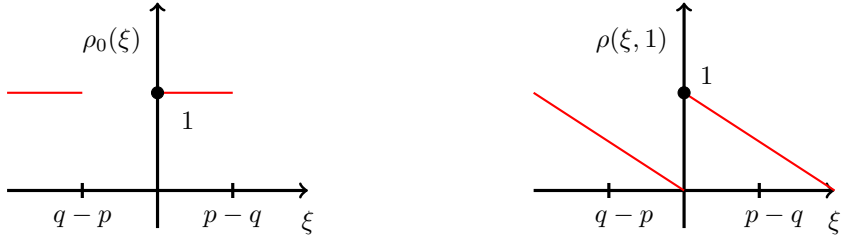
\begin{figure}
 \begin{center}
   \begin{tikzpicture}
       \draw (0.4,0.7) node[anchor=south]{\small{$1$}};
  \draw [very thick, ->] (0,-0.5) -- (0,2.5);
    \draw (-0.1,2) node[anchor=east] {\small{$\rho_{0}(\xi)$}};
    \draw[red,thick ] (-2,1.3) -- (-1,1.3);
      \draw[thick,red] (0,1.3) -- (1,1.3);
\draw[very thick] (-1,-0.1)--(-1,0.1)  ;
  \draw (-1,-0.6) node[anchor=south]{\small{$q-p$}};
  \draw[very thick] (1,-0.1)--(1,0.1)  ;
  \draw (1,-0.6) node[anchor=south]{\small{$p-q$}};
\filldraw(0,1.3) circle(0.08cm);
   \draw [very thick, ->] (-2,0) -- (2,0) node[below=4pt] {\small{$\xi$}};

\begin{scope}[xshift=7cm]
 \draw [very thick,->] (-2,0) -- (2,0) node[below=4pt] {\small{$\xi$}};
  \draw [very thick,->] (0,-0.5) -- (0,2.5);
    \draw (0.3,1.3) node[anchor=south]{\small{$1$}};
    \draw[red,thick] (-2,1.3) -- (0,0);
    \draw[red,thick] (0,1.3)--(2,0);
    \draw[very thick] (-1,-0.1) -- (-1,0.1);
       \draw (-0.8,-0.6) node[anchor=south]{\small{$q-p$}};
    \draw[very thick] (1,-0.1) -- (1,0.1);
       \draw (1.2,-0.6) node[anchor=south]{\small{$p-q$}};
\filldraw(0,1.3) circle(0.08cm); 
    \draw (-0.1,2) node[anchor=east] {\small{$\rho(\xi,1)$}};
   \end{scope}
   \end{tikzpicture}    \end{center} \caption{ Left: The macroscopic   initial particle density $\rho_{0}$ of    the initial configurations \eqref{IC} and, after rescaling $\xi\to \xi/(p-q),$ \eqref{IC2}.
     Right: The large time particle density $\rho$  for $\rho_{0}$. At the origin, $\rho$ jumps from $0$ to $1$, and $\rho(-\varepsilon,1),1-\rho(\varepsilon,1)>0$ for any $\varepsilon>0$. }\label{curvdens}
\end{figure}

Let us  define the Tracy-Widom $F_{\GUE}$ distribution function  which appears in our main results: It originates in random matrix theory \cite{TW94} and is given by 
\begin{equation*}\label{FGUE2}
F_{\GUE}(s)=\sum_{n=0}^{\infty} \frac{(-1)^{n}}{n!}  \int_{s}^{\infty}\dx x_{1}\ldots  \int_{s}^{\infty}\dx x_{n}\det(K_{2}(x_{i},x_{j})_{1\leq i,j\leq n}),\end{equation*}
where $K_{2}(x,y)$ is the Airy kernel  $K_{2}(x,y)=\frac{Ai(x)Ai^{\prime}(y)-Ai(y)Ai^{\prime}(x)}{x-y},x\neq y, $ defined for $x=y$ by continuity and $Ai$ is the Airy function.  

 The following Theorem, proven in Section \ref{adapt}, is our first main result.  We  get the same fluctuations also in a single limit, see Theorem \ref{GUEGUE2}. 
Let us emphasize that $M\in \Z_{\geq 1}$ denotes the label of a particle and that $M$ remains fixed as we send $t\to \infty.$ 

\begin{tthm}\label{GUEGUE}
Consider ASEP with the initial data \eqref{IC} and $C=C(M)$ as in \eqref{C}.
Then 
\begin{equation}\label{Jamie}
\lim_{M \to \infty} \lim_{t \to \infty}\Pb\left( x_{M+\lambda M^{1/3}}(t)\geq -\xi M^{1/3}\right)=F_{\GUE}(-\lambda)F_{\GUE}(\xi-\lambda)
\end{equation}
for $\lambda,\xi \in \R$. Furthermore,  we have for $s \in \R \setminus\{0\}$
\begin{equation}\label{Thehound}
\lim_{M \to \infty} \lim_{t \to \infty}\Pb\left( x_{M+\lambda M^{1/3}}(t)\geq -s\sqrt{p-q}t^{1/2}M^{-1/6}\right)=F_{\GUE}(s-\lambda)\mathbf{1}_{\{s>0\}}.
\end{equation}
\end{tthm}

A few remarks are in order.  It was essential to scale $C$ as in \eqref{C}, since e.g. for $C$ fixed, the double limit \eqref{Jamie} would be equal to zero. 

Note  that  the limits  \eqref{Jamie} and \eqref{Thehound} are consistent in the sense that there is a continuous  transition
\begin{equation*}
\lim_{s \searrow 0}F_{\GUE}(s-\lambda)\mathbf{1}_{\{s>0\}}=\lim_{\xi\to + \infty}F_{\GUE}(-\lambda)F_{\GUE}(\xi-\lambda).
\end{equation*}
The convergence \eqref{Thehound} means that to the left of the shock, $ x_{M+\lambda M^{1/3}}(t)$ fluctuates like the $(M+\lambda M^{1/3})$-th particle of ASEP with step initial data, i.e.  $ x_{M+\lambda M^{1/3}}(t)$ fluctuates as if there was no shock.

Inside the shock, the fluctuation behavior of   $ x_{M+\lambda M^{1/3}}(t)$ changes : the  $\xi M^{1/3}$ term  in \eqref{Jamie} is  the usual KPZ $1/3$ fluctuation exponent, whereas the particle number $M+\lambda M^{1/3}$ in \eqref{Jamie} represents the degenerated correlation length known from shocks in TASEP:
one takes $M+\lambda M^{1/3}$ rather than $M+\lambda M^{2/3}$ ($2/3$ being the typical KPZ correlation exponent), and $x_{M+\lambda M^{2/3}}$ no longer converges to $F_{\GUE}\times F_{\GUE}$. See also the comparison between hard and non-hard shock fluctuations in Section \ref{compare}.







Next we come to a shock where the convergence to  $F_{\GUE}\times F_{\GUE}$ happens in a single limit: For this, we consider a \textit{different} initial data, and to highlight this difference, we will write  $X_{n}(t)$ for the position of the particle with label $n$ at time $t,$ rather than $x_{n}(t)$ as before. To define this new  initial data, let $\nu\in (0,1)$ and  set  
\begin{equation}\label{IC2}
X_{n}(0)=
\begin{cases}
-n-\lfloor t- 2 t^{\nu/2+1/2}\rfloor   \quad &\mathrm{for} \,  n \geq 1 \\
-n\quad &\mathrm{for}\, -\lfloor t  - 2 t^{\nu/2+1/2}\rfloor  \leq n \leq 0.
\end{cases}
\end{equation}
The initial data \eqref{IC2} creates, upon rescaling time as  $t/(p-q)$ and sending $t\to \infty$,  the same macroscopic density profile as the initial data \eqref{IC} (see Figure \ref{curvdens}). 
However, in the following Theorem \ref{GUEGUE2}, we have a   shock region  of size $\mathcal{O}(t^{\nu/3}),$ whereas in Theorem  \ref{GUEGUE}, the shock region is of size $\mathcal{O}(M^{1/3})$.
The convergence to Tracy-Widom distributions thus happens in a single $t\to \infty$ limit:
\begin{tthm}\label{GUEGUE2}
Consider ASEP with initial data \eqref{IC2}. Let
$\lambda,\xi \in \R,$ and $s\in \R\setminus\{0\}.$ Then, for $\nu\in(0,3/7)$ we have  
\begin{align}\label{jaja}
&\lim_{t \to \infty} \Pb\left(X_{t^{\nu}+\lambda t^{\nu/3}}(t/(p-q))\geq -\xi t^{\nu/3}\right)=F_{\GUE}(-\lambda)F_{\GUE}(\xi-\lambda)
\\&\lim_{t \to \infty} \Pb\left(X_{t^{\nu}+\lambda t^{\nu/3}}(t/(p-q))\geq -s t^{1/2-\nu/6}\right)=F_{\GUE}(s-\lambda)\mathbf{1}_{\{s>0\}}.\label{jaja2}
\end{align}
\end{tthm}
The proof of Theorem \ref{GUEGUE2} is given in Section \ref{adapt} and has the same structure as the proof of Theorem \ref{GUEGUE}, but some extra care has to be taken because of the parameter $\nu$. The technical  restriction $\nu<3/7$ comes  into play to show that certain ASEPs have enough time to mix to equilibrium as well as for the convergence \eqref{jaja2}, see  the explanations around \eqref{ineqq2},  \eqref{aligned} (in Section \ref{lowsec}) and \eqref{EM2} (in Section \ref{adapt})  for details. 
 It is however unlikely that the value $3/7$ represents a real threshold. In \eqref{jajanee2} of Section \ref{adapt},    we prove \eqref{jaja}   for all $\nu\in (0,1)$ for TASEP, and show that $F_{\GUE}(-\lambda)F_{\GUE}(\xi-\lambda)$ is an upper bound in  \eqref{jaja} for the general ASEP.

\subsection{Comparison with non-hard shocks in TASEP}\label{compare}
In the special case of totally asymmetric exclusion ($p=1$), we can transit between the fluctuations at a non-hard shock, and the hard shock fluctuations of Theorems \ref{GUEGUE} and \ref{GUEGUE2}.

Specifically, consider for TASEP  the initial data
\begin{equation}\label{tildex}
\tilde{x}_{n}(0)=
\begin{cases}
-n-\lfloor \beta t\rfloor   \quad &\mathrm{for} \,  n \geq 1 \\
-n\quad &\mathrm{for}\, -\lfloor \beta t\rfloor  \leq n \leq 0,\\
\end{cases}
\end{equation}
where $\beta \in (0,1)$. This is  the shock which was studied in Corollary 2.7 of \cite{FN14} by coupling TASEP with last passage percolation.
For the $\tilde{x}_{n}$, there is  a shock at the origin where the density jumps from $(1-\beta)/2$ to $(1+\beta)/2$.
We  show that there is a smooth transition 
between the fluctuations at the hard shock in \eqref{Jamie} and the  shock created by the initial data \eqref{tildex}:
\begin{cor}\label{TASEPcor}
Consider TASEP with the initial configurations $\tilde{x}_{n},x_{n}$ with $C$ as in \eqref{C}. Then we have
\begin{align*}
\lim_{\beta \nearrow 1}\lim_{t \to \infty}&\Pb\left(\tilde{x}_{\frac{(1-\beta)^{2}}{4}t+\lambda \frac{(1-\beta)^{2/3}}{2^{2/3}}t^{1/3}}(t)\geq -\xi  \frac{(1-\beta)^{2/3}}{2^{2/3}}t^{1/3}\right)
\\&=\lim_{M \to \infty} \lim_{t \to \infty}\Pb\left( x_{M+\lambda M^{1/3}}(t)\geq -\xi M^{1/3}\right)=F_{\GUE}(-\lambda)F_{\GUE}(\xi-\lambda).
\end{align*}
\end{cor}
\begin{proof}

By  Corollary  2.7 of  \cite{FN14}, we have 
\begin{equation*}
\lim_{t\to\infty}\Pb\left(\tilde{x}_{\frac{(1-\beta)^{2}}{4}t+\lambda  t^{1/3}}(t)\geq -\xi t^{1/3}\right) =F_\GUE \left(\frac{\xi-\lambda/\rho_1}{\sigma_1}\right)F_\GUE \left(\frac{\xi-\lambda/\rho_2}{\sigma_2}\right)
\end{equation*}
with $\rho_1=\frac{1-\beta}{2}$, $\rho_2=\frac{1+\beta}{2}$, $\sigma_1=\frac{(1+\beta)^{2/3}}{2^{1/3}(1-\beta)^{1/3}}$, and $\sigma_2=\frac{(1-\beta)^{2/3}}{2^{1/3}(1+\beta)^{1/3}}$.
By a simple calculation and the continuity of the $F_\GUE$ distribution function 
this gives 
\begin{align*}
\lim_{\beta \nearrow 1}\lim_{t \to \infty}&\Pb\left(\tilde{x}_{\frac{(1-\beta)^{2}}{4}t+\lambda \frac{(1-\beta)^{2/3}}{2^{2/3}}t^{1/3}}(t)\geq -\xi  \frac{(1-\beta)^{2/3}}{2^{2/3}}t^{1/3}\right)
=F_{\GUE}(-\lambda)F_{\GUE}(\xi-\lambda),
\end{align*}
finishing the proof by using  Theorem \ref{GUEGUE}.
\end{proof}
\subsection{Graphical construction and applications}\label{BC}
In this section, we recall the graphical construction of ASEP, the basic coupling, the approximation of infinite ASEPs by finite ones, and we prove a correlation inequality using a result of  Harris \cite{Har77}, which we formulate as Proposition \ref{FKG}.
While it is  possible that Proposition \ref{FKG} is somewhere in the literature, we do not know where, hence we provide a full proof.

The graphical construction of particle systems goes back to Harris \cite{Har78}. 
We take a collection $(P^{i,i+1},P^{i,i-1},i\in\Z)$ of independent Poisson processes constructed on some probability space $(\hat{\Omega},\mathcal{A}, \mathbb{P} )$ . The processes $(P^{i,i+1},i\in\Z)$ have rate $p\in (1/2,1],$  and the processes $(P^{i,i-1},i\in\Z)$ have rate $q=1-p$. We denote by $P^{i,i\pm 1}_{t}$ the value of the Poisson process at time $t$. 
 Due to the independence of the Poisson processes, for almost every $\omega\in \hat{\Omega}$ there is a doubly infinite sequence $(i_{n},n\in \Z)$ of integers such that 
\begin{equation}\label{in}
\cdots <i_{n-2}< i_{n-1}<i_{n}<i_{n+1}<i_{n+2}<\cdots,
\end{equation}
and $P^{i_{n}+1,i_{n}}_{t}(\omega)=P^{i_{n},i_{n}+ 1}_{t}(\omega)=0,n \in \Z$. 

We construct an ASEP starting from $\zeta_{0}\in \{0,1\}^{\Z}$ by the following two rules: When the Poisson process $P^{i,i+1}$ makes a jump, and  there is a particle  at $i$ and a hole at $i+1,$ then the particle at $i$ and the hole at $i+1$ exchange positions. Similarly, when the Poisson process  $P^{i,i-1}$ has a jump, and there is  a particle at $i$ and a hole at $i-1,$  the particle at $i$ and the hole at $i-1$ exchange positions. Note that under these two rules,  no particle enters or exists through the sites $i_{n},n\in \Z,$  during the time interval $[0,t],$ so it suffices to apply the two rules inside each of the finite boxes $\{i_{n},\ldots,i_{n+1}\}$  to construct the process up to time $t$.  Inside each of the boxes, there are a.s. finitely many jumps and  no two jumps happen at the same time, so the graphical construction is well-defined.   

The graphical construction  allows us to couple different ASEPs together. Given a sequence of initial configurations $\zeta_{0}^{K}\in \{0,1\}^{\Z},K\in\Z_{\geq 1},$ we use the same collection of Poisson processes $(P^{i,i+1},P^{i,i-1},i\in\Z)$ 
to graphically construct all ASEPs $(\zeta_{t}^{K},t\geq 0, K\in\Z_{\geq 1}).$
This coupling is henceforth referred to as the \textit{basic coupling}.
An important consequence of the basic coupling is that if $(\zeta_{t}^{K},\zeta_{t}^{K^{\prime}},t\geq 0)$ are coupled via the basic coupling,
and  if for all $j\in \Z$ we have  $\zeta_{0}^{K}(j)\leq \zeta_{0}^{K^{\prime}}(j),$ it follows  that for all $t\geq 0$ we have   $\zeta_{t}^{K}(j)\leq \zeta_{t}^{K^{\prime}}(j)$
for all $j\in \Z.$

Let us now draw some consequences from the graphical construction and the basic coupling.

\subsubsection{Approximation by finite ASEPs}\label{finite}
Here we note  that we may arbitrarily well approximate the position of a particle in an infinite ASEP by that of a particle in a finite ASEP. 
 To see this, use the graphical construction with Poisson processes $(P^{i,i+1},P^{i,i-1},i\in\Z)$ to obtain an ASEP $(\zeta_{t},t\geq 0)$  starting from $\zeta_{0}$.
 We now graphically construct a finite  ASEP   $(\zeta^{L}_{t},t\geq 0)$ on $\{0,1\}^{\{-L,\ldots,L\}}, $ where particles move in $\{-L,\ldots,L\}$ and start from
 $\zeta^{L}_{0}(j):=\zeta_{0}(j),|j|\leq L$.  The graphical construction of the finite ASEP uses a finite subfamily of the Poisson processes used for the construction of $(\zeta_{t},t\geq 0)$, namely the Poisson processes
  $(P^{i,i+1},P^{i,i-1},|i|\leq L-1).$ All the other Poisson processes  $(P^{i,i+1},P^{i,i-1},|i|\geq L)$ are not in use.
  
  Let now $r\in \{i:\zeta_{0}(i)=1\}$  and choose $L>r$.
Denote by $\sigma_{r}(t)$ the position at time $t$ of the particle initially at $r$ in the infinite ASEP $(\zeta_{t},t\geq 0)$. Likewise, we denote  by $\sigma_{r}^{L}(t)$ the position at time $t$ of the particle initially at $r$ in the finite ASEP $(\zeta_{t}^{L},t\geq 0)$. In the graphical construction, $\sigma_{r}(t)$ a.s. depends only on what happens in a finite box which contains $r$. Thus, in particular, as $L\to\infty$,  $\sigma_{r}^{L}(t)$ and $\sigma_{r}(t)$ coincide: We have that 
\begin{equation}\label{Lfinite}
\lim_{L\to+\infty}\Pb(\sigma_{r}^{L}(t)=\sigma_{r}(t))=1.
\end{equation}

 \subsubsection{Harris correlation inequality}
 Intuitively, it is clear  that the event that a particle is to the right of some position $s_{1}$ should be positively correlated with the event that another particle is to the right of some position $s_{2}$. The result that we use to show this  goes back to Harris \cite{Har77}. The version in which we use Harris's result  is  Theorem 2.14 of \cite{Li85b}, see in particular  the remarks  made directly after  Theorem 2.14 of \cite{Li85b} which apply in our case.
 The statement is as follows.
  \begin{prop} \label{FKG}Let  $(\zeta_{t},t\geq 0)$ be an ASEP with a deterministic initial configuration  that has at least two particles, i.e. $|\{i:\zeta_{0}(i)=1\}|\geq 2$. Let $r,r^{\prime}\in \{i:\zeta_{0}(i)=1\}$
 and $r\neq r^{\prime}$. We denote by $\sigma_{r}(t), \sigma_{r^{\prime}}(t)$ the position at time $t$ of the particles initially at $r,r^{\prime}$.  Then we have for $s_{1},s_{2}\in \Z$ 
 \begin{equation}
 \Pb\left(\{\sigma_{r}(t)\geq s_{1}\}\cap\{\sigma_{r^{\prime}}(t)\geq s_{2}\}\right)\geq 
 \Pb(\{\sigma_{r}(t)\geq s_{1}\}) \Pb (\{\sigma_{r^{\prime}}(t)\geq s_{2}\} ).
 \end{equation}
 \end{prop}
 \begin{proof}
 We will show the statement for finite ASEPs,  and then use the approximation \eqref{Lfinite}. So let $L\in \Z_{\geq 1}$  with $L>r,r^{\prime}$ and consider the 
 finite ASEP $(\zeta_{t}^{L},t\geq 0)$ on $\{0,1\}^{\{-L,\ldots,L\}}$ starting from 
 $\zeta_{0}^{L}(j)=\zeta_{0}(j),|j|\leq L$ that was constructed via coupling  in Subsection \ref{finite}.   We denote  by $\sigma_{r}^{L}(t), \sigma_{r^{\prime}}^{L}(t)$ the position  at time $t$ of the particles initially at $r,r^{\prime}$
 in the ASEP  $(\zeta_{t}^{L},t\geq 0)$.
 
 The general statement that we wish to apply is  Theorem 2.14 of \cite{Li85b}.

 For this, we need to define a partial order $\triangleright$ on $\{0,1\}^{\{-L,\ldots,L\}}.$ The  partial order $\triangleright$  we use is 
 \begin{equation}
 \zeta \triangleright  \zeta^{\prime} \iff \sum_{j=l}^{L}\zeta(j)\leq \sum_{j=l}^{L}\zeta^{\prime}(j) \quad \mathrm{for \,\, all\,}l\in \{-L,\ldots,L\}.
\end{equation}  
We denote by $\Pb_{t}$ the distribution of $\zeta_{t}^{L}$ on   $\{0,1\}^{\{-L,\ldots,L\}}.$ We say that $f:\{0,1\}^{\{-L,\ldots,L\}}\to \R$ is increasing w.r.t.  $\triangleright$ if $\zeta \triangleright \zeta^{\prime}$ implies 
$f(\zeta)\leq f(\zeta^{\prime})$. For increasing functions $f,g$   we say that they are positively correlated under $\Pb_{t}$  if
\begin{equation}\label{poscor}
\int fg  \dx\mathbb{P}_{t}\geq \int f  \dx\mathbb{P}_{t}\int  g\dx \mathbb{P}_{t}.
\end{equation}

 Theorem 2.14 of \cite{Li85b} shows that \eqref{poscor} holds for all increasing functions if the following three conditions are met: i) increasing functions are positively correlated under $\Pb_{0}$ ii) ASEP is attractive w.r.t. the partial order $\triangleright$  and iii) ASEP can only jump from a state $\zeta$ to a state $\hat{\zeta}$  that satisfies $\zeta \triangleright  \hat{\zeta}$  or 
 $\hat{\zeta} \triangleright  \zeta.$
 Condition i) is trivially true because $\mathbb{P}_{0}$ is a dirac measure.
  To check ii), it suffices to show that if $\zeta_{0}\triangleright\zeta_{0}^{\prime},$ then under the basic coupling  for all $t\geq 0$ we have $\zeta_{t}\triangleright\zeta_{t}^{\prime}.$ Assume to the contrary 
that  $T:=\inf\{\ell: \zeta_{\ell}\triangleright\zeta_{\ell}^{\prime} \mathrm{\, does \,not\,hold}\}<\infty$. Then, there is an $l$ such that 
  \begin{equation}\label{aligneddd}
  \begin{aligned}
 &\sum_{j=l}^{L}\zeta_{T}(j) > \sum_{j=l}^{L}\zeta_{T}^{\prime}(j)
  \\&  \sum_{j=l}^{L}\zeta_{T^{-}}(j)= \sum_{j=l}^{L}\zeta_{T^{-}}^{\prime}(j),
  \end{aligned}
  \end{equation}
   where $\zeta_{T^{-}}^{L}:=\lim_{t\nearrow T}\zeta_{t}^{L}$ denotes the state just before $T$.
One of the two possibilities for \eqref{aligneddd} to be true  is that a particle in the $\zeta^{\prime}-$ process  jumps from $l$ to $l-1$ at time $T,$ but this jump does not happen in the $\zeta-$process. This  requires $\zeta_{T^{-}}^{\prime}(l-1)<\zeta_{T^{-}}(l-1)$ which in turn  implies 
  \begin{align*}
  \sum_{j=l-1}^{L}\zeta_{T^{-}}(j) =1+\sum_{j=l-1}^{L}\zeta_{T^{-}}^{\prime}(j),
  \end{align*}
  contradicting $\zeta_{T^{-}}\triangleright\zeta_{T^{-}}^{\prime}.$
 The other possibility for \eqref{aligneddd} to be true is that 
 a particle in the $\zeta-$ process  jumps from $l-1$ to $l$ at time $T,$ but this jump does not happen in the $\zeta^{\prime}-$process. But this requires $\zeta_{T^{-}}^{\prime}(l)>\zeta_{T^{-}}(l)$ which in turn  implies 
  \begin{align*}
 \sum_{j=l+1}^{L}\zeta_{T^{-}}(j)= 1+\sum_{j=l+1}^{L}\zeta_{T^{-}}^{\prime}(j),
  \end{align*}  again 
 contradicting $\zeta_{T^{-}}\triangleright\zeta_{T^{-}}^{\prime}.$ This shows ii).

Let us check iii).
Assume  $\tau$ is a random time point where a particle jumps. We easily see that  we always have 
\begin{equation}
\zeta_{\tau^{-}}^{L}\triangleright \zeta_{\tau}^{L} \mathrm{\quad or\quad}
\zeta_{\tau}^{L}\triangleright\zeta_{\tau^{-}}^{L},
\end{equation}
depending on whether at time $\tau$ a particle jumped to the right, in which case $\zeta_{\tau^{-}}^{L}\triangleright \zeta_{\tau}^{L}, $ or it jumped to the left, in which case $\zeta_{\tau}^{L}\triangleright\zeta_{\tau^{-}}^{L}. $ So iii) holds.

In conclusion, we know that \eqref{poscor} holds, it remains to choose $f$ and $g$. For this, we define the integers
\begin{equation}
R_{1}=\sum_{j=r}^{L}\zeta_{0}(j)\quad R_{2}=\sum_{j=r^{\prime}}^{L}\zeta_{0}(j).
\end{equation}
Define the sets $A_{1}=\{\zeta:\sum_{j=s_{1}}^{L}\zeta(j)\geq R_{1}\},$
$A_{2}=\{\zeta:\sum_{j=s_{2}}^{L}\zeta(j)\geq R_{2}\}$ and $f=\mathbf{1}_{A_{1}},g=\mathbf{1}_{A_{2}}$.
We may then conclude 
\begin{equation}
\begin{aligned}\label{necnec}
 \Pb\left(\{\sigma_{r}^{L}(t)\geq s_{1}\}\cap\{\sigma_{r^{\prime}}^{L}(t)\geq s_{2}\}\right)=\int fg \dx \mathbb{P}_{t}&\geq \int f  \dx\mathbb{P}_{t}\int  g\dx \mathbb{P}_{t} 
 \\&=\Pb(\{\sigma_{r}^{L}(t)\geq s_{1}\}) \Pb (\{\sigma_{r^{\prime}}^{L}(t)\geq s_{2}\} ).
\end{aligned}
\end{equation}
Sending $L\to +\infty$  in \eqref{necnec} and using \eqref{Lfinite} finishes the proof.

 \end{proof}

\subsection{Method of proof}\label{method} Here we outline the strategy to prove Theorem \ref{GUEGUE}; the same strategy is  used to prove Theorem \ref{GUEGUE2} as well, but the parameter $\nu$  needs to be dealt with.
We  mostly  use  probabilistic tools such as couplings and  bounds on mixing times in this paper.

To prove the convergence \eqref{Jamie} of Theorem \ref{GUEGUE}, 
we provide an upper and a lower bound for\[ \lim_{t \to \infty}\Pb\left( x_{M+\lambda M^{1/3}}(t)\geq -\xi M^{1/3}\right)\]  and show that the two bounds converge, as $M\to\infty$, to 
\[F_{\GUE}(-\lambda)F_{\GUE}(\xi-\lambda).\] For the upper bound, we define the  initial data
\begin{align*}
&x_{n}^{A}(0)=
-n-\lfloor (p-q)(t-Ct^{1/2})\rfloor  \quad &&\mathrm{for} \quad   n \geq 1 
\\& x_{n}^{B}(0)=
-n  \quad  &&\mathrm{for} \quad   n \geq  -\lfloor (p-q)(t-Ct^{1/2})\rfloor,
\end{align*}
and denote by $ (\eta_{\ell}^{A})_{\ell  \geq 0},(\eta_{\ell}^{B})_{\ell  \geq 0}$ the ASEPs started from these initial data.
The intuition behind these two initial data  is that the particle  $ x_{n} $ should behave like  $x_{n}^{A}$ if it does not enter the shock, and like $ x_{n}^{B}$ if  it enters. 
This intuition is correct for TASEP, but in ASEP this only provides an upper bound:
Defining  for $n\geq 1$ the minimum 
\begin{equation}\label{yn}
y_{n}(t)=\min\{x_{n}^{A}(t),x_{n}^{B}(t)\},
\end{equation}
then   $y_{n}$ and $x_{n}$ are related under the basic coupling  as follows.
\begin{prop}\label{sepprop} Let  the ASEPs $(\eta_{t},\eta^{A}_{t},\eta^{B}_{t},t\geq 0)$ 
be coupled through the basic coupling given in Section \ref{BC}. 
In TASEP, we have for $n\geq 1$
\begin{equation}\label{TY}
y_{n}(t)=x_{n}(t),
\end{equation}
whereas for ASEP we have 
\begin{equation}\label{Y}
y_{n}(t)\geq x_{n}(t).
\end{equation}\end{prop}
\begin{proof}
The identity \eqref{TY} is an application   of the coupling provided in Lemma 2.1 of \cite{Sep98c}. Indeed, consider the shifted step initial data $\zeta^{l}=\mathbf{1}_{\Z\leq l}$, and denote for $i\geq 0$  by $\zeta^{l}(i,t)$ the position at time $t$ 
of the particle  initially at position $l-i$ in the TASEP started from $\zeta^{l}$. Let now the TASEPs started from $\eta_{0},\zeta^{l},l\in \Z,$ be coupled through the basic coupling. 
Then  Lemma 2.1 of \cite{Sep98c} shows that for TASEP we have
\begin{equation*}
x_{n}(t)=\min_{N(t)\leq i\leq n}\zeta^{x_{i}(0)}(n-i,t),
\end{equation*}
where for brevity  we defined $N(t):= -\lfloor (p-q)(t-Ct^{1/2})\rfloor$. Then we have $x_{n}^{A}(t)=\zeta^{x_{1}(0)}(n-1,t)$ and $x_{n}^{B}(t)=\zeta^{x_{N(t)}(0)}(n-N(t),t)$. From the relation $\zeta^{l}(i,t)\geq \zeta^{l+1}(i+1,t),$ valid under the basic coupling,  it is easy to see
that
\begin{equation*}
\zeta^{x_{1}(0)}(n-1,t)=\min_{1\leq i\leq n}\zeta^{x_{i}(0)}(n-i,t)\quad\quad  \zeta^{x_{N(t)}(0)}(n-N(t),t)=\min_{N(t)\leq i\leq 0}\zeta^{x_{i}(0)}(n-i,t),
\end{equation*}
implying \eqref{TY}.

To prove the inequality \eqref{Y}, let us first show $x_{n}^{A}(t)\geq x_{n}(t)$.
Assume to the contrary $x_{n}^{A}(t)< x_{n}(t).$ Since $\eta_{t}(j)\geq \eta_{t}^{A}(j),j\in \Z,$ this implies that for all $n^{\prime}\in \{n,n+1,\ldots\},$ there is an $n^{\prime \prime}>n^{\prime}$ such that 
$x_{n^{\prime}}^{A}(t)=x_{n^{\prime\prime}}(t)$. However, since $x_{n^{\prime}}^{A}(0)=x_{n^{\prime}}(0)$ and $\lim_{N\to+\infty}\Pb(x_{N}(0)=x_{N}(t))=1,$ we cannot have 
$x_{n^{\prime}}^{A}(t)=x_{n^{\prime\prime}}(t)$ for infinitely many integers $n^{\prime}<n^{\prime\prime}$. Thus, we have $x_{n}^{A}(t)\geq x_{n}(t)$.
To see that  $x_{n}^{B}(t)\geq x_{n}(t)$ holds, assume again to the contrary that  $x_{n}^{B}(t)< x_{n}(t).$  This implies that
\begin{equation*}
\sum_{j=x_{n}(t)}^{\infty}\eta_{t}(j)>\sum_{j=x_{n}(t)}^{\infty}\eta_{t}^{B}(j),
\end{equation*}
contradicting $\eta_{t}(j)\leq \eta_{t}^{B}(j),j\in \Z$.
\end{proof}
That we only have an  inequality  in \eqref{Y} in the general asymmetric case is one of the main reasons why  proving Theorems \ref{GUEGUE} and \ref{GUEGUE2} is harder for   ASEP  than TASEP.

 To get an upper bound in \eqref{Jamie}, we show in Corollary  \ref{GUEGUEcor} that $x_{n}^{A}(t),x_{n}^{B}(t)$ decouple as $t\to\infty,$ and that \eqref{Jamie} holds with $x_{M+\lambda M^{1/3}}(t)$    replaced by $y_{M+\lambda M^{1/3}}(t)$. A key tool we use is the slow decorrelation method \cite{Fer08}, \cite{CFP10b}, as well as proving that certain particles remain in disjoint space-time regions,  see Section \ref{decsec}. This gives the desired upper bound for 
 $ \lim_{t \to \infty}\Pb\left( x_{M+\lambda M^{1/3}}(t)\geq -\xi M^{1/3}\right)$.

 For the lower bound, let us describe a general strategy which in this paper is applied in Proposition \ref{X0} and Theorem \ref{lowerthm}.
 Suppose for some (arbitrary) ASEP particle $z_{N}$ 
 we wish to prove
 \begin{equation*}\label{genlowbnd}
 \lim_{N\to \infty}\lim_{t\to\infty}  \Pb (z_{N}(t)\geq R(N))\geq F 
 \end{equation*}
where  $R(N) \in \Z,$ and $ F\in [0,1],$ e.g. $F=F_{\GUE}(-\lambda)F_{\GUE}(\xi-\lambda)$.
 The way we proceed is as follows: For $\chi>0$ we  construct an event $\mathcal{E}_{t-t^{\chi}}$ which depends only on what happens in ASEP during  $[0,t-t^\chi].$
 We assume that  $ \Pb( \mathcal{E}_{t-t^{\chi}}) \geq F(N,t)$ with $ \lim_{N\to \infty}\lim_{t\to\infty}F(N,t)=F$ and furthermore, 
 we assume to  have a relation 
 \begin{equation}\label{allgemeine}
  \mathcal{E}_{t-t^{\chi}}\subseteq \{ z_{N}(t)\geq \tensor[^-]{x}{_{N}}(t)\}.
 \end{equation}
 Here, $\tensor[^-]{x}{_{N}}(t)$ is a particle in a countable state space ASEP starting at time $t-t^{\chi}$ from a deterministic initial configuration; in particular 
 $\tensor[^-]{x}{_{N}}(t)$ is independent of $\mathcal{E}_{t-t^{\chi}}$.
 The point is that for $\chi>0$ sufficiently large, this ASEP has  enough time to  come very close to equilibrium,  and the equilibrium is such that this implies 
 \begin{equation}\label{such that}
 \Pb( \tensor[^-]{x}{_{N}}(t)\geq R(N))\geq 1-\varepsilon(N,t),
 \end{equation}
 with $\varepsilon(N,t)$ going to zero as $N,t\to \infty$.
We can then compute
  \begin{align*}
 \Pb (z_{N}(t)\geq R(N))\geq   \Pb (   \mathcal{E}_{t-t^{\chi}}\cap \{z_{N}(t)\geq R(N)\} )&\geq  \Pb (\{  \tensor[^-]{x}{_{N}} (t)\geq R(N)\} \cap   \mathcal{E}_{t-t^{\chi}})
 \\&= \Pb (\{  \tensor[^-]{x}{_{N}} (t)\geq R(N)\} ) \Pb (  \mathcal{E}_{t-t^{\chi}}),
 \end{align*}
and from this we obtain 
 \begin{align*}
 \lim_{N\to\infty}\lim_{t \to\infty}\Pb (z_{N}(t)\geq R(N))\geq   \lim_{N\to\infty}\lim_{t \to\infty} \Pb (\{  \tensor[^-]{x}{_{N}} (t)\geq R(N)\} ) \Pb (  \mathcal{E}_{t-t^{\chi}})\geq F.
 \end{align*}

Let us briefly describe how the general strategy is used for the last step of the proof of Theorem \ref{GUEGUE}.  In Section \ref{lowsec} we  consider for $\chi<1/2$  and $\delta>0$ small  the event   \begin{equation}
\label{passauf}\mathcal{E}_{t-t^{\chi}}=\{x_{0}(t-t^{\chi})\geq M+(\lambda-\xi)M^{1/3}\}
\cap \{x_{M+\lambda M^{1/3}}(t-t^{\chi})>-t^{ \delta} \}.
\end{equation}
Furthermore, the $F,R$ from the general strategy will be given by  $F=F_{\GUE}(-\lambda)F_{\GUE}(\xi-\lambda)$ and $R=-\xi M^{1/3}$.
An important part of the work is to   show that asymptotically $ \lim_{M\to \infty}\lim_{t\to\infty} \Pb( \mathcal{E}_{t-t^{\chi}}) \geq F$ . 
For this, we  apply  the Harris  inequality of Proposition \ref{FKG} to \eqref{passauf}. Then we  again have to employ the general strategy in Proposition \ref{X0} to show 
\begin{equation*}
\lim_{M \to \infty}\lim_{t\to \infty}\Pb(x_{0}(t-t^{\chi})>M+M^{1/3}(\lambda-\xi))\geq F_{\GUE}(\xi-\lambda).
\end{equation*}
Finally, various  couplings arguments allow us to show \begin{equation*} \lim_{M\to \infty}\lim_{t\to\infty} \Pb(x_{M+\lambda M^{1/3}}(t-t^{\chi})>-t^{ \delta} )= F_{\GUE}(-\lambda).\end{equation*}
 This shows $ \lim_{M\to \infty}\lim_{t\to\infty} \Pb( \mathcal{E}_{t-t^{\chi}}) \geq F.$
 
 The next step is to start at time $t-t^{\chi}$ a countable state space  ASEP from
\begin{equation*}
\tensor[^-]{\eta}{_{}}(j)=\mathbf{1}_{\{-\lfloor t^{\delta}\rfloor,\ldots,-\lfloor t^{\delta}\rfloor+M+\lambda M^{1/3}-1\}}(j)+\mathbf{1}_{\{j\geq M+(\lambda-\xi) M^{1/3}\}}(j),\quad j\in \Z.
\end{equation*}
We show the relation \eqref{allgemeine}, specifically, on  $\mathcal{E}_{t-t^{\chi}}$,  $ z_{N}(t)= x_{M+\lambda M^{1/3}}(t)$ is bounded from below by the leftmost particle  $ \tensor[^-]{x}{_{M+\lambda M^{1/3}}}(t)$ of $\tensor[^-]{\eta}{_{t}}.$ 
In Proposition \ref{DOIT}, we give lower bounds for the position of the leftmost particle in countable state space ASEPs. For this we use the results  on hitting/mixing times  of \cite{BBHM}, for recent major progress on this topic,  see \cite{LL19}.
From Proposition \ref{DOIT}  we get that  $(\tensor[^-]{\eta}{_{\ell}})_{\ell\geq t-t^{\chi}}$  has enough time to  come very close to its equilibrium  during $[t-t^{\chi},t]$ and that  in equilibrium,  $ \tensor[^-]{x}{_{M+\lambda M^{1/3}}}\geq -\xi M^{1/3}$ holds  with very high probability.
Furthermore,  $(\tensor[^-]{\eta}{_{\ell}})_{\ell\geq t-t^{\chi}}$ is independent of the event \eqref{passauf}, which will allow us to get  the desired 
lower bound. See Section \ref{lowsec} for the details.
\subsection{Outline}
In Section \ref{convsec} we collect convergence results for ASEP with step initial data 
that we need as input and prove the convergence to $ F_{\GUE}$ in a double limit (see Proposition \ref{FGUE}). In Section \ref{mixsec}, we first bound the position of the leftmost particle in a reversed step initial data. Then, as key tool, we control the position 
of particles using bounds on the mixing time (see Proposition \ref{DOIT}). In Section \ref{decsec} we employ the slow decorrelation method and bounds on particle positions  to prepare the proof of the decoupling of  $x_{n}^{A}(t),x_{n}^{B}(t)$ given in Section \ref{upsec}.  In Section \ref{upsec}, an upper bound for the limit \eqref{Jamie} is proven using this decoupling, and  \eqref{Thehound} is proven also.  Section \ref{lowsec} gives the required lower bound for \eqref{Jamie}. In Section \ref{adapt} we can then quickly prove Theorem \ref{GUEGUE}.
The proof of Theorem \ref{GUEGUE2} is identical in structure to the proof of Theorem \ref{GUEGUE}. In Sections \ref{convsec} - \ref{lowsec} we only deal with the initial data \eqref{IC}, in Section \ref{adapt} we explain how to adapt the results of Sections 
 \ref{convsec} - \ref{lowsec} to   prove Theorem \ref{GUEGUE2}.

\section{Convergence Results for ASEP with step initial data} \label{convsec}
 Let us start by defining the distribution functions which will appear throughout  this paper. 
\begin{defin}[ \cite{TW08b},\cite{GW90}] Let $s \in \R,M \in \Z_{\geq 1}$. We define, for $p\in (1/2,1),$  
\begin{equation} \label{def1}
F_{M,p}(s)=\frac{1}{2 \pi i} \oint  \frac{\dx \lambda}{\lambda} \frac{\det(I-\lambda K)}{\prod_{k=0}^{M-1}(1-\lambda (q/p)^{k})}
\end{equation}
where $K= \hat{K}\mathbf{1}_{(-s,\infty)}$  and 
$\hat{K}(z,z^{\prime})=\frac{p}{\sqrt{2 \pi}}e^{-(p^{2}+q^{2})(z^{2}+z^{\prime 2})/4+pqzz^{\prime}}$ and the integral is taken over a counterclockwise oriented contour enclosing the
singularities $\lambda=0,\lambda=(p/q)^{k},k=0,\ldots,M-1$.  For $p=1$, we define
\begin{equation*}
F_{M,1}(s)=\Pb\left(\sup_{0=t_{0}<\cdots <t_{M}=1}\sum_{i=0}^{M-1}[B_{i}(t_{i+1})-B_{i}(t_{i})]\leq s    \right),
\end{equation*}
where $B_{i},i=0,\ldots,M-1$ are independent standard Brownian motions. \end{defin}

 It follows from \cite{Bar01}, Theorem 0.7 that $F_{M,1}$ equals the distribution function of the largest eigenvalue of a $M \times M\, \GUE$  matrix. What is important to us here is that $F_{M,p}$ arises as limit law in ASEP,
 a result we cite in Theorem \ref{convthm} below (Theorem \ref{convthm} also vindicates our common denomination $F_{M,p}$ for all $p\in (1/2,1]$ even though $F_{M,1}$ looks different from $F_{M,p},p<1.$)

 For $p<1$, the following Theorem  was shown in \cite{TW08b}, Theorem 2, for TASEP, the result follows e.g.  from \cite{GW90}, Corollary 3.3, see Remark 3.1 of \cite{GW90} for further references. An alternative characterization of the limit \eqref{ASEPlimit}
 was given in \cite{BO17}, Proposition 11.1.
\begin{tthm}[Theorem 2 in \cite{TW08b},  Corollary 3.3 in \cite{GW90}]\label{convthm}
Consider ASEP with step initial data $x_{n}^{\mathrm{step}}(0)=-n,n\geq 1$.
Then  for every fixed $M\geq 1$ we have that 
\begin{equation}\label{ASEPlimit}
\lim_{t \to \infty}\Pb\left(x_{M}^{\mathrm{step}}(t)\geq (p-q)(t-st^{1/2})\right)=F_{M,p}(s).
\end{equation}
\end{tthm}

We need to show that the  distribution function $F_{M,p}$ converges to $F_{\GUE}$ in the right scaling. To show this, we do not actually use the explicit formula \eqref{def1}, but rather the alternative characterization provided in \cite{BO17}, and the following proof is similar to that of Theorem 11.3 of \cite{BO17}.
\begin{prop}\label{FGUE}

For any fixed  $s \in \R$  we have 
\begin{equation}\label{two1}
\lim_{M \to \infty}F_{M,p}\left(\frac{2\sqrt{M}+sM^{-1/6}}{\sqrt{p-q}}\right)=F_{\GUE}(s).
\end{equation}

\end{prop}

\begin{proof}
For TASEP, this follows  from \cite{Bar01}, Theorem 0.7  which shows that $F_{M,1}$ equals the distribution function of the largest eigenvalue of a $M \times M\, \GUE$  matrix.
For $p<1$,  we use the notation and the methods  provided in \cite{BO17}.
Define for $r\in \R$ a $\Z_{\geq 0}$-valued random variable $\xi_{r}$ via
\begin{equation*}
\Pb(\xi_{r}\geq  M)=F_{M,p}\left(\frac{\sqrt{2}}{\sqrt{p-q}}r\right), \quad M\in \Z_{\geq1},
\end{equation*}
(note that by definition $\Pb(\xi_{r}\geq 0)=1$ and it follows directly from \eqref{ASEPlimit} that $\Pb(\xi_{r}\geq M)\leq \Pb(\xi_{r}\geq M-1)$).

Define, for $\mathfrak{q} \in (0,1),\zeta\in \mathbb{C}\setminus\{ -\mathfrak{q}^{-j},j\geq0\},$ 
\begin{equation*}
\mathcal{L}^{(\mathfrak{q})}_{\xi_{r}}(\zeta)=\mathbb{E}\left(     \prod_{i\geq 1}\frac{1}{1+\zeta \mathfrak{q}^{\xi_{r}+i}} \right),
\end{equation*} and note that $\mathcal{L}^{(\mathfrak{q})}_{\xi_{r}}$ characterizes the law of $\xi_{r}$.

A  random point process $P$ on $\Z_{\geq 0}$ is a probability measure on the subsets of $\Z_{\geq 0}$.  We define the random variable $\mathcal{X},$ which maps   each subset $X\subseteq \Z_{\geq 0}$ to $\C$ via
\[\mathcal{X}(X)= \prod_{x \in X}\frac{1}{1+\zeta \mathfrak{q}^{x}}.\]
Denoting expectation w.r.t. $P$ as $E,$ we define 
\begin{equation*}
\mathfrak{L}^{(\mathfrak{q})}_{P}(\zeta)=E(\mathcal{X}).
\end{equation*}


Finally, let $\mathrm{DHermite}^{+}(r)$ be (one of the two variants of) the  discrete Hermite ensemble, a determinantal point process on $\Z_{\geq 0}$ introduced in Section 3.2 of \cite{BO17}.
By Proposition 11.1 of \cite{BO17},
we have
\begin{equation*}
\mathcal{L}^{(\mathfrak{q})}_{\xi_{r}}(\zeta)=\mathfrak{L}^{(\mathfrak{q})}_{\mathrm{DHermite}^{+}(r)}(\zeta).
\end{equation*}
Let  $(r_{n})_{n\geq 1}$ be a sequence in $\R$ with $r_n\to +\infty$. We use the notion of asymptotic equivalence, as defined in Definition 11.7 of \cite{BO17}.
Now by Corollary 5.7 of \cite{B17}, the sequence 
$F_{n}(y)=\mathfrak{L}^{(\mathfrak{q})}_{\mathrm{DHermite}^{+}(r_n)}(\mathfrak{q}^{y})$
is asymptotically equivalent to $-\min\mathrm{DHermite}^{+}(r_{n})$.
On the other hand, by Example 5.5 of \cite{B17}, $(\xi_{r_{n}})_{n\geq 1}$ is asymptotically equivalent to $ \mathcal{L}^{(\mathfrak{q})}_{-\xi_{r_{n}}}(\mathfrak{q}^{y})$.
Since $ \mathcal{L}^{(\mathfrak{q})}_{-\xi_{r_{n}}}=\mathfrak{L}^{(\mathfrak{q})}_{-\mathrm{DHermite}^{+}(r)}$ and since being asymptotically equivalent is  a transitive relation,
it follows that $(\min\mathrm{DHermite}^{+}(r_{n}))_{n\geq 1}$ and $(\xi_{r_{n}})_{n\geq 1}$ are asymptotically equivalent. This in particular implies 
\begin{align*}
\lim_{M \to \infty}\Pb(\xi_{\sqrt{2M}+sM^{-1/6}}\geq M)=\lim_{M \to \infty}\Pb(\min\mathrm{DHermite}^{+}(\sqrt{2M}+sM^{-1/6})\geq M).
\end{align*}
Now by the duality of the discrete and continuous Hermite ensemble (Theorem 3.7 in \cite{BO17})
we have 
\begin{align*}
\lim_{M \to \infty}\Pb(\min\mathrm{DHermite}^{+}&(\sqrt{2M}+sM^{-1/6})\geq M)\\&=\lim_{M \to \infty}\Pb(\mathrm{CHermite}(M)\mathrm{\,has\, no \,particles\,in\,}[\sqrt{2M}+sM^{-1/6},+\infty)),
\end{align*}
where $\mathrm{CHermite}(M)$ is the continuous $M-$particle Hermite ensemble, i.e. the determinantal point process on $\R$ with correlation kernel
\begin{equation*}
K_{M}(x,y)=e^{-x^{2}/2}e^{-y^{2}/2}\sum_{n=0}^{M-1}\frac{\hat{H}_{n}(x)\hat{H}_{n}(y)}{||H_{n}||^{2}}
\end{equation*}
w.r.t. the Lebesgue measure, the $(\hat{H}_{n})_{n\geq0}$ being  the Hermite polynomials, which are orthogonal on  $L^{2}(\R,e^{-x^{2}}\dx x)$  and have leading coefficients $2^{n},n\geq 1$. The convergence 
\begin{equation*}
\lim_{M \to \infty}\Pb(\mathrm{CHermite}(M)\mathrm{\,has\, no \,particles\,in\,}[\sqrt{2M}+sM^{-1/6},+\infty))=F_{\GUE}(\sqrt{2}s)
\end{equation*}
is a classical result and is e.g. proved as Theorem 3.14 in great detail in the textbook \cite{AGZ10}, see Chapters 3.2 and 3.7 therein; note that the definition of Hermite polynomials differs slightly in \cite{AGZ10}.
\end{proof}

\section{Bounds on particle positions using stationary measures}\label{mixsec}
In this Section, we provide bounds of the leftmost particle of several countable state space ASEPs.
A prominent role will play what we call reversed step initial data: 
Define for $Z\in \Z$  the ASEP  $ ( \eta^{-\mathrm{step}(Z)}_{\ell})_{\ell\geq 0} $ started from the reversed step initial data
\begin{equation*}
\eta^{-\mathrm{step}(Z)}(j)=\mathbf{1}_{\Z_{\geq Z}}(j),
\end{equation*}
for $Z=0$ we simply write $\eta^{-\mathrm{step}}$.

We start by bounding the position of the leftmost particle of $\eta^{-\mathrm{step}(Z)}_{\ell}$. 
\begin{prop}\label{block}
Consider  ASEP with reversed step  initial data $x_{-n}^{\mathrm{-step(Z)}}(0)=n+Z,n\geq 0$ and let $\delta>0$. Then  there is a $t_0$ such that for $t>t_{0},R\in \Z_{\geq 1}$ and constants $C_{1},C_{2}$ (which depend on $p$) we have
\begin{align}\label{bms2}
&\Pb\left(  x_{0}^{\mathrm{-step}(Z)}(t)<Z-R\right)\leq C_{1}e^{-C_{2}R}
\\&\label{bms}
\Pb\left( \inf_{0\leq \ell \leq t} x_{0}^{\mathrm{-step}(Z)}(\ell)<Z-t^{\delta}\right)\leq C_{1}e^{-C_{2}t^{\delta}}.
\end{align}
\end{prop}
\begin{proof}
By translation invariance, we may w.l.o.g. set $Z=0$.
We prove the proposition by comparing the reversed step initial data $\eta^{-\mathrm{step}}=\mathbf{1}_{\Z_{\geq0}}$ with an invariant blocking measure $\mu$.
The measure $\mu$ on $\{0,1\}^{\Z}$ is the product measure with marginals
\begin{equation}\label{blockm}
\mu(\{\eta:\eta(i)=1\})=\frac{c(p/q)^{i}}{1+c(p/q)^{i}}
\end{equation}
with  $c>0$  a free parameter we choose later. It is well known that $\mu$ is invariant for ASEP \cite{Lig76}.

Let $(\eta^{\mathrm{block}}_{s})_{s\geq 0}$  be the ASEP started from the initial distribution $\mu$, and denote by $x_{0}^{\mathrm{block}}(s)$ the position of the leftmost particle of $\eta^{\mathrm{block}}_{s}$.
Let $(\eta^{\mathrm{-step}}_{s})_{s\geq 0}$  be the ASEP started from the reversed step initial data  $\eta^{-\mathrm{step}}.$ We let $ (\eta^{\mathrm{block}}_{s})_{s\geq 0},(\eta^{\mathrm{-step}}_{s})_{s\geq 0}$ evolve together under the basic coupling.

Let us  first prove that for any fixed  $0\leq \ell \leq t$ 
\begin{equation}\label{equation}
\Pb\left(  x_{0}^{\mathrm{-step}}(\ell)<-R\right)\leq  \frac{1}{c(1-q/p)}+c\left(\frac{q}{p}\right)^{R}\frac{1}{1-q/p}.
\end{equation}
To prove \eqref{equation}, consider the partial order on $\{0,1\}^{\Z}$ given by 
\begin{equation*}
\eta \leq \eta^{\prime}\iff \eta(i)\leq\eta^{\prime}(i) \,\mathrm{\,for\,\,all\,\,}i\in \Z
\end{equation*}
and use $\eta \not \leq \eta^{\prime}$ as short hand for the statement  that $\eta \leq \eta^{\prime}$ does not hold.
We can now bound
\begin{equation}
\begin{aligned}\label{un}
\Pb\left( x_{0}^{\mathrm{-step}}(\ell)<-R\right)&\leq \Pb\left( \{ x_{0}^{\mathrm{-step}}(\ell)<-R\}\cap \{\eta_{\ell}^{\mathrm{block}}\geq \eta_{\ell}^{\mathrm{-step}}\}\right)
\\&+\Pb(\eta_{\ell}^{\mathrm{block}}\not \geq \eta_{\ell}^{\mathrm{-step}}).
\end{aligned}
\end{equation}
Let us bound the two  terms on the R.H.S. of \eqref{un}.
By attractivity  of ASEP,
\begin{equation}\label{trois}
\Pb(\eta_{\ell}^{\mathrm{block}}\not \geq \eta_{\ell}^{\mathrm{-step}})\leq \Pb(\eta_{0}^{\mathrm{block}}\not \geq \eta^{\mathrm{-step}}_{0}).
\end{equation}
Using the simple estimates $\log(1+\varepsilon)\leq \varepsilon$ and $\exp(-\varepsilon)\geq 1-\varepsilon$ for $\varepsilon\geq 0$
we obtain
\begin{equation}
\begin{aligned}\label{quatre}
\Pb(\eta_{0}^{\mathrm{block}}\not \geq \eta^{\mathrm{-step}}_{0})&=1-\exp\left(-\sum_{i=0}^{\infty}\log(1+(q/p)^{i}/c)\right)
\\&\leq1-\exp\left(- 1/(c(1-q/p))\right)
\\&\leq \frac{1}{c(1-q/p)}.
\end{aligned}
\end{equation}
Furthermore, we have\begin{equation}
\begin{aligned}\label{deux}
\Pb\left( \{ x_{0}^{\mathrm{-step}}(\ell)<-R\}\cap \{\eta_{\ell}^{\mathrm{block}}\geq \eta_{\ell}^{\mathrm{-step}}\}\right)&\leq \Pb\left(  x_{0}^{\mathrm{block}}(\ell)<-R\right)
\\&=\Pb\left(  x_{0}^{\mathrm{block}}(0)<-R\right),
\end{aligned}\end{equation}
where the identity in \eqref{deux} follows from the invariance of $\mu$.
By a computation very similar to \eqref{quatre} we obtain
\begin{equation}\label{cinq}
\begin{aligned}
\Pb\left(  x_{0}^{\mathrm{block}}(0)<-R\right)\leq c\left(\frac{q}{p}\right)^{R}\frac{1}{1-q/p}.
\end{aligned}
\end{equation}

This proves \eqref{equation} by combining the inequalities \eqref{un},\eqref{deux},\eqref{trois},\eqref{quatre} and \eqref{cinq}.
If we choose $c=(p/q)^{R/4}$ in \eqref{equation}, we obtain \eqref{bms2}.

Since \eqref{equation} does not depend on $\ell$, we obtain for $R=t^{\delta}/2$
\begin{equation*}\label{equation2}
\Pb\left( \bigcup_{\ell=1,2,\ldots,\lfloor t\rfloor } \{x_{0}^{\mathrm{-step}}(\ell)<-t^{\delta}/2\}\right)\leq t\left(\frac{1}{c(1-q/p)}+c\left(\frac{q}{p}\right)^{t^{\delta}/2}\frac{1}{1-q/p}\right).
\end{equation*}
Note further that for the event 
\begin{equation}\label{event}
\bigcap_{\ell=1,2,\ldots,\lfloor t \rfloor } \{x_{0}^{\mathrm{-step}}(\ell)\geq -t^{\delta}/2\}\cap\{ \inf_{0\leq \ell \leq t} x_{0}^{\mathrm{-step}}(\ell)<-t^{\delta}\}
\end{equation}
to hold, $x_{0}^{\mathrm{-step}}$ would need to make $t^{\delta}/2$ jumps to the left in a time interval $[\ell,\ell+1],\ell=0,\ldots,t-1$. 
For any fixed time interval $[\ell,\ell+1]$ the probability that  $x_{0}^{\mathrm{-step}}$ makes at least  $k$ jumps to the left is bounded by the probability that a rate $q$ Poisson process makes at least $k$ jumps in a unit time interval.
In particular, the probability that  $x_{0}^{\mathrm{-step}}$ makes $t^{\delta}/2$ jumps to the left during $[\ell,\ell+1]$ may be  bounded by $e^{-t^{\delta}/2}$ for $t\geq t_{0},$ and $t_{0}$ sufficiently large. Since there are $t$ such intervals, we see that the probability of the event \eqref{event} is bounded by $te^{-t^{\delta}/2}$.
So in total we obtain 
\begin{align}\nonumber
\Pb\left( \inf_{0\leq \ell \leq t} x_{0}^{\mathrm{-step}}(\ell)<-t^{\delta}\right)&\leq    \Pb\left( \bigcup_{\ell=1,2,\ldots,\lfloor t\rfloor} \{x_{0}^{\mathrm{-step}}(\ell)<-t^{\delta}/2\}\right)
\\&\nonumber+\Pb\left(\bigcap_{\ell=1,2,\ldots,\lfloor t\rfloor} \{x_{0}^{\mathrm{-step}}(\ell)\geq -t^{\delta}/2\}\cap\{ \inf_{0\leq \ell \leq t} x_{0}^{\mathrm{-step}}(\ell)<-t^{\delta}\}\right)
\\&\label{ja!}\leq   t\left( \frac{1}{c(1-q/p)}+c\left(\frac{q}{p}\right)^{t^{\delta}/2}\frac{1}{1-q/p}+e^{-t^{\delta}/2}\right).
\end{align}
Choosing $c=(p/q)^{t^{\delta}/4}$ in \eqref{ja!} we obtain \eqref{bms} for $t $ sufficiently large.
\end{proof}

The particle configuration  $\eta^{-\mathrm{step}(Z)}$  lies in the countable set
\begin{equation*}\label{OZ}
\Omega_{Z}=\big\{\eta\in\{0,1\}^{\Z}:\sum_{j=-\infty}^{Z-1}\eta(j)=\sum_{j=Z}^{\infty}1-\eta(j)<\infty\big\},
\end{equation*}
and an ASEP started from $\Omega_{Z}$ remains in $\Omega_{Z}$ for all times.
Furthermore, an  ASEP started from an element of $\Omega_{Z}$ has as unique invariant measure
\begin{equation*}
\mu_{Z}=\mu(\cdot | \Omega_{Z}),
\end{equation*}
with $\mu$ the blocking measure \eqref{blockm} ($\mu$ depends on the parameter $c$, but $\mu_{Z}$ does not).
On $\Omega_{Z}$ we define the partial order 
\begin{equation}\label{preceq}
\eta\preceq\eta^{\prime}\iff \sum_{j=r}^{\infty}1-\eta^{\prime}(j)\leq \sum_{j=r}^{\infty}1-\eta(j)\quad \mathrm{for \, all\,}r\in \Z.
\end{equation}
While this order is only partial, all $\eta\in \Omega_{Z} $ satisfy
\begin{equation*}
\eta\preceq\eta^{-\mathrm{step}(Z)}=\mathbf{1}_{\Z_{\geq Z}}.
\end{equation*}
The following Lemma will be used repeatedly to bound the position of the  leftmost particle of ASEPs in $\Omega_{Z}.$
\begin{lem}\label{lem}
Let   $\eta,\eta^{\prime}\in \Omega_{Z}$ and consider the basic coupling of two ASEPs $(\eta_{\ell})_{\ell\geq 0},(\eta^{\prime}_{\ell})_{\ell\geq 0}$ started from  $\eta_{0}=\eta,\eta^{\prime}_{0}=\eta^{\prime}$.   For $s\geq 0,$ denote   by $x_{0}(s),x_{0}^{\prime}(s)$ the position of the leftmost particle of $\eta_{s},\eta^{\prime}_{s}.$ 
Then, if $\eta\preceq\eta^{\prime}$, we have $x_{0}(s)\leq x_{0}^{\prime}(s)$.
\end{lem}
\begin{proof}
If $\eta\preceq\eta^{\prime}$, then  $\eta_s\preceq\eta^{\prime}_s,$ hence it suffices to prove the lemma for $s=0$. As  $\eta,\eta^{\prime}\in \Omega_{Z},$ there is an $R_{0} \in \Z$ such that 
\begin{equation*}
\sum_{j=R_{0}}^{\infty}1-\eta^{\prime}(j)= \sum_{j=R_{0}}^{\infty}1-\eta(j)=0.
\end{equation*}
Note $x_{0}(0), x_{0}^{\prime}(0)\leq R_{0}. $
Now imagine $x_{0}(0)>x_{0}^{\prime}(0)$.  The set $\{x_{0}(0),\ldots,R_{0}\}$ contains $R_{0}-Z+1$ particles from $\eta$. On the other hand, the set $\{x_{0}(0),\ldots,R_{0}\}$ contains less than $R_{0}-Z+1$
particles from $\eta^{\prime}$ because otherwise $x_{0}^{\prime}(0)\in  \{x_{0}(0),\ldots,R_{0}\}, $ contradicting $x_{0}(0)>x_{0}^{\prime}(0)$.
But this implies
\begin{align*}
\sum_{j=x_{0}(0)}^{\infty}1-\eta^{\prime}(j)&=\sum_{j=x_{0}(0)}^{R_{0}-1}1-\eta^{\prime}(j)+\sum_{j=R_{0}}^{\infty}1-\eta^{\prime}(j)\\&> \sum_{j=x_{0}(0)}^{R_{0}-1}1-\eta(j)+\sum_{j=R_{0}}^{\infty}1-\eta(j)=\sum_{j=x_{0}(0)}^{\infty}1-\eta(j),
\end{align*}
contradicting $\eta\preceq\eta^{\prime}$.
\end{proof}

The next result  will be very important to prove a lower bound for \eqref{Jamie} in Section \ref{lowsec}. We use bounds from \cite{BBHM} on the time it takes an ASEP started from $\Omega_{Z}$  to hit the maximal state  $\eta^{-\mathrm{step}(Z)}$.
As shown in \cite{BBHM}, these bounds  on hitting times imply that the mixing time of biased card shuffling on $N$ cards   is  $\mathcal{O}(N).$ In our context, we will use these bounds to   control the position of the leftmost particle in an ASEP started from a specific initial data.
\begin{prop}\label{DOIT}
Let $a,b,N\in\Z$ and $a\leq b\leq N.$
Consider the ASEP $(\eta^{a,b,N}_{\ell})_{\ell\geq 0}$ with initial data 
\begin{equation}\label{abN}
\Omega_{N-b+a}\ni\eta^{a,b,N}_{0}=\mathbf{1}_{\{a,\ldots,b\}}+\mathbf{1}_{\Z_{\geq N+1}}
\end{equation}
and denote by $x_{0}^{a,b,N}(s)$ the position of the leftmost particle of  $\eta^{a,b,N}_{s}$.
Let $\mathcal{M}=\max\{b-a+1,N-b\}$ and $\varepsilon>0$. Then there are constants $C_{1},C_{2}$ (depending on $p$) and a constant $K$ (depending on $p,\varepsilon$)
so that for $s>K\mathcal{M}$ and $R\in \Z_{\geq 1}$ 
\begin{equation*}
\Pb\left(x_{0}^{a,b,N}(s)<N-b+a-R\right)\leq \frac{\varepsilon}{\mathcal{M}}+ C_{1}e^{-C_{2}R}.
\end{equation*}
\end{prop}
\begin{proof}
Consider first the case $N-b\leq b-a+1$. Let $x\geq 0$ be such that $N-b+x=b-a+1$ (i.e. $x=2b-a+1-N$). Consider an ASEP $(I^{b-a+1}_{\ell})_{\ell\geq 0}$ started from 
\begin{equation*}
I^{b-a+1}_{0}=\mathbf{1}_{\{a-x,\ldots,b-x\}}+\mathbf{1}_{\Z_{\geq N+1}}.
\end{equation*}
Clearly $I^{b-a+1}_{0}\preceq \eta^{a,b,N}_{0}$.
Let us denote the position of the leftmost particle of $I^{b-a+1}_{\ell}$ by $\mathcal{I}_{0}(\ell)$.
Then with  $I_{b-a+1}$ defined in (4) of \cite{BBHM}  we have 
\begin{equation}\label{mossel}
1-I^{b-a+1}_{0}(j)=I_{b-a+1}(j+x-b-1), \quad j\in \Z.
\end{equation}
Consider the hitting time
\begin{equation*}
\mathfrak{H}(I^{b-a+1})=\inf\{\ell:  I^{b-a+1}_{\ell}=\eta^{-\mathrm{step}(N-b+a)}\}.
\end{equation*}
By \eqref{mossel}, Theorem 1.9 of \cite{BBHM} directly gives that for every $\varepsilon>0$ there is a constant $K$  such that 
\begin{equation*}
\Pb(\mathfrak{H}(I^{b-a+1})\geq K(b-a+1))\leq \frac{\varepsilon}{b-a+1}.
\end{equation*}
Hence we may conclude for $s>K(b-a+1)$ 
\begin{align}\label{firsteq}
\Pb\left(x_{0}^{a,b,N}(s)<N-b+a-R\right)&\leq \Pb\left(\mathcal{I}_{0}(s)<N-b+a-R, \mathfrak{H}(I^{b-a+1})\leq K(b-a+1)\right)
\\&+\Pb\left( \mathfrak{H}(I^{b-a+1})>K(b-a+1)\right)\nonumber
\\&\leq \Pb\left(x_{0}^{-\mathrm{step}(N-b+a)}(s)<N-b+a-R\right)+\frac{\varepsilon}{b-a+1}\label{2ndeq}
\\& \leq C_{1}e^{-C_{2}R} +\frac{\varepsilon}{b-a+1},\label{lasteq}
\end{align}
where \eqref{firsteq} follows from Lemma \ref{lem}, \eqref{lasteq} follows from   Proposition \ref{block}, and  for  \eqref{2ndeq} we used that, when  $\eta^{-\mathrm{step}(N-b+a)}_{s}, I^{b-a+1}_{s}$ are coupled via the basic coupling, then for $s\geq \mathfrak{H}(I^{b-a+1})$ we have $\eta^{-\mathrm{step}(N-b+a)}_{s}=I^{b-a+1}_{s}$ (this is  so because 
 $I^{b-a+1}_{s} \preceq \eta^{-\mathrm{step}(N-b+a)}_{s} $ for all $s\geq 0,$ hence $\eta^{-\mathrm{step}(N-b+a)}=I^{b-a+1}_{ \mathfrak{H}(I^{b-a+1})}\preceq  \eta^{-\mathrm{step}(N-b+a)}_{\mathfrak{H}(I^{b-a+1})}, $ implying $I^{b-a+1}_{ \mathfrak{H}(I^{b-a+1})}=  \eta^{-\mathrm{step}(N-b+a)}_{\mathfrak{H}(I^{b-a+1})}$ and hence $\eta^{-\mathrm{step}(N-b+a)}_{s}=I^{b-a+1}_{s}$ for all $s\geq \mathfrak{H}(I^{b-a+1})$).
 
 If $N-b\geq b-a+1,$ we proceed similarly : Let $\tilde{x}\geq0$ so that $b-a+1+\tilde{x}=N-b$ and consider
 \begin{equation*}
 \tilde{I}^{N-b}_{0}=\mathbf{1}_{\{a,\ldots,b+\tilde{x}\}}+\mathbf{1}_{\Z_{>N+\tilde{x}}}.
 \end{equation*}
 Then 
 \begin{equation*}
1-\tilde{I}^{N-b}_{0}(j)=I_{N-b}(j-N+b-a), \quad j\in \Z
\end{equation*}
and we have \begin{equation*}
\Pb(\mathfrak{H}(\tilde{I}^{N-b}_{0})\leq K(N-b))>1-\frac{\varepsilon}{N-b}.
\end{equation*}
The remaining part of the proof is identical.
\end{proof}

\section{Slow decorrelation and asymptotic independence}\label{decsec}

Let us recall the two  initial data that were defined in Section \ref{method} for $C\in\R$:
\begin{equation}
\begin{aligned}\label{AB}
&x_{n}^{A}(0)=
-n-\lfloor (p-q)(t-Ct^{1/2})\rfloor  \quad &&\mathrm{for} \quad   n \geq 1 
\\& x_{n}^{B}(0)=
-n  \quad  &&\mathrm{for} \quad   n \geq  -\lfloor (p-q)(t-Ct^{1/2})\rfloor.
\end{aligned}
\end{equation}
In this section, we consider  fixed  $M\in \Z_{\geq 1},\lambda \in \R$ so that 
$M+\lambda M^{1/3}\geq 1$ 
and   employ the slow decorrelation methodology to prepare the proof of the decoupling of $x_{M+\lambda M^{1/3}}^{A}(t)$ and $x_{M+\lambda M^{1/3}}^{B}(t)$ given  in Section \ref{upsec}.

We start by recalling  the following elementary Lemma. We denote by $"\Rightarrow"$ convergence in distribution.
\begin{lem}\label{elemlem}
Let $(X_{n})_{n\geq 1}, (\tilde{X}_{n})_{n\geq 1}$ be sequences of random variables such that $X_{n}\geq \tilde{X}_{n}$. Let $X_{n } \Rightarrow D, \tilde{X}_{n}\Rightarrow D,$ where $D$ is a probability distribution. Then $X_{n}- \tilde{X}_{n}\Rightarrow0$.
\end{lem}
The following is our slow decorrelation statement.
\begin{prop}\label{prop0}
Let $\kappa \in (0,1)$ and $\varepsilon>0,\lambda \in \R$.  We have 
\begin{equation*}
\lim_{t\to \infty}\Pb\left(\left|x_{M+\lambda M^{1/3}}^{A}(t)-x_{M+\lambda M^{1/3}}^{A}(t-t^{\kappa})-(p-q)t^{\kappa}\right|\geq \varepsilon t^{1/2}\right)=0.
\end{equation*} 
\end{prop}

\begin{proof}
We may assume w.l.o.g. that $\lambda=0.$
Consider an ASEP with step initial data which starts at time $t-t^{\kappa}$ and has its rightmost particle at position $x_{M}^{A}(t-t^{\kappa})$:
Set  $\tilde{\eta}_{t-t^{\kappa}}(i):=\mathbf{1}_{\{i\leq x_{M}^{A}(t-t^{\kappa})\}}(i)$ and denote by $(\tilde{\eta}_{t-t^{\kappa}+s},s \geq 0)$ 
the ASEP which starts  at time $t-t^{\kappa}$ from   $\tilde{\eta}_{t-t^{\kappa}}$ .
Denote by $\tilde{x}_{1}(t)$ the position  of the rightmost particle of  $\tilde{\eta}_{t}$.
Then we have 
\begin{equation} \label{upperbound}
x_{M}^{A}(t)\leq x_{M}^{A}(t-t^{\kappa})+\tilde{x}_{1}(t)- x_{M}^{A}(t-t^{\kappa})
\end{equation}
Now
\begin{equation}\label{dashier}\tilde{x}_{1}(t)- x_{M}^{A}(t-t^{\kappa})-1=^{d}x_{1}^{\mathrm{step}}(t^{\kappa})\end{equation}
where $=^{d}$ denotes equality in distribution and $x_{1}^{\mathrm{step}}(t^{\kappa})$ is the position at time $t^{\kappa}$ of the rightmost particle 
in ASEP started with step initial data $x_{n}^{\mathrm{step}}(0)=-n,n\geq 0$.
Now by Theorem \ref{convthm} we have in particular that $\{x_{1}^{\mathrm{step}}(t^{\kappa})-(p-q)t^{\kappa})t^{-\kappa/2}\}_{t\geq 0}$ is tight,
which together with  \eqref{dashier} implies
\begin{equation}\label{eins}
\lim_{t \to \infty} \Pb\left(|\tilde{x}_{1}(t)- x_{M}^{A}(t-t^{\kappa})-(p-q)t^{\kappa} |t^{-1/2}\geq \varepsilon/2\right)=0.
\end{equation}

Again using  Theorem \ref{convthm}, we obtain
\begin{align}\label{zwei}
&\frac{x_{M}^{A}(t-t^{\kappa})+(p-q)t^{\kappa}}{t^{1/2}(p-q)}\Rightarrow F_{M,p}(\cdot+C)
\quad  \frac{x_{M}^{A}(t)}{t^{1/2}(p-q)}\Rightarrow F_{M,p}(\cdot+C).
\end{align}
So by \eqref{eins},
\begin{align*}
&\frac{x_{M}^{A}(t-t^{\kappa})+(p-q)t^{\kappa}}{t^{1/2}(p-q)}+      \frac{\tilde{x}_{1}(t)-x_{M}^{A}(t-t^{\kappa})-(p-q)t^{\kappa}}{t^{1/2}(p-q)}    \Rightarrow F_{M,p}(\cdot+C).
\end{align*}
Thus we can apply Lemma \ref{elemlem} to \eqref{upperbound}, 
which then implies 
\begin{align}\label{stochconv}
\frac{x_{M}^{A}(t-t^{\kappa})+(p-q)t^{\kappa} -\tilde{x}_{1}(t)}{t^{1/2}(p-q)}- \frac{x_{M}^{A}(t)-\tilde{x}_{1}(t)}{t^{1/2}(p-q)} \Rightarrow 0,
\end{align}
using \eqref{eins}, and \eqref{stochconv} is the desired statement.
\end{proof}

The next proposition shows that $x_{M+\lambda M^{1/3}}^{A}(t-t^{\kappa}),x_{M+\lambda M^{1/3}}^{B}(t) $ are asymptotically (i.e. in the $t\to\infty$ limit) independent for $\kappa>1/2$. This asymptotic independence does not require us to send $M\to \infty$, so it holds for arbitrary values of the constant $C$ of \eqref{AB}.
\begin{prop}\label{indepprop}
Consider the ASEPs $(\eta_{\ell}^{A}, \eta_{\ell}^{B},\ell \geq 0) $ started from \eqref{AB} for arbitrary $C\in \R$ and under the basic coupling.
Let $\kappa \in (1/2,1),R \in \Z, \tilde{C}, \lambda  \in \R.$ Then 
\begin{align*}
\lim_{t\to \infty}\Pb(&\min\{x_{M+\lambda M^{1/3}}^{A}(t-t^{\kappa})+(p-q)(t^{\kappa}+\tilde{C}t^{1/2}),x_{M+\lambda M^{1/3}}^{B}(t)\}\geq -R)\\&=\begin{cases} F_{M+\lambda M^{1/3},p}(C+\tilde{C})&\mathrm{for} \,  R \geq M+\lambda M^{1/3} \\
F_{M+\lambda M^{1/3},p}(C+\tilde{C})F_{M+\lambda M^{1/3}-R,p}(C) &\mathrm{for} \,  R < M+\lambda M^{1/3}.
\end{cases}
\end{align*} 
\end{prop}
\begin{proof}
Again we assume w.l.o.g. $\lambda=0$.

 Consider first the case $R<M$. Define the collection of holes
 \begin{equation}\label{holes}
H_{n}^{B}(0)=n+\lfloor (p-q)(t-Ct^{1/2})\rfloor, \quad n\geq 1.
\end{equation}
Note the $H_{n}^{B}$ perform an ASEP with (shifted) step initial data, where the holes jump to the right with probability $q<1/2$ and to the left with probability $p=1-q$.
The relation between holes and particles  is 
\begin{equation*}
\{H_{M-R}^{B}(t)<-R\}=\{x_{M}^{B}(t)\geq -R\}.
\end{equation*}
The limit law of both $H_{M-R}^{B}(t)$ and $x_{M}^{A}(t-t^{\kappa})$  is  given by Theorem \ref{convthm} (recall again the shift by $Ct^{1/2}$ in both \eqref{AB},\eqref{holes}) :
\begin{equation}
\begin{aligned}\label{need}
&\lim_{t \to \infty}\Pb (x_{M}^{A}(t-t^{\kappa})+(p-q)(t^{\kappa}+\tilde{C}t^{1/2})\geq -R)=F_{M,p}(C+\tilde{C})
\\&\lim_{t \to \infty}\Pb (H_{M-R}^{B}(t)<-R)=F_{M-R,p}(C).
\end{aligned}
\end{equation}
The basic idea of the proof is that the trajectories  $x_{1}^{A}(s),0\leq s\leq t-t^{\kappa},$  and  $H_{1}^{B}(s),0\leq s\leq t,$ stay in disjoint space-time regions as $t\to \infty$, and hence  $x_{n}^{A}(t-t^{\kappa}),n \geq 1,$ and 
$H_{n}^{B}(t),n \geq 1,$ are independent asymptotically.

Define now $\tilde{\eta}^{A}_0:=\eta_{0}^{A}$ . Let $0<\varepsilon<\kappa-1/2$. 
Graphically construct $(\tilde{\eta}^{A}_s)_{s\geq 0}$ just like $(\eta^{A}_s)_{s\geq 0},$ using the same Poisson processes,  with the difference that all jumps
in the space-time region
\begin{align*}
\{(i,s)\in \Z\times \R_{+}:i\geq-(p-q)t^{\kappa}/4, 0\leq s\leq t-t^{\kappa}\}
\end{align*}
are suppressed. Denote by $\tilde{x}_{M}^{A}(t-t^{\kappa})$ the position of the $M$-th  particle (counted from right to left) of $\tilde{\eta}^{A}_{t-t^{\kappa}}$.

Likewise, define  $\tilde{\eta}^{B}_0=\eta_{0}^{B}$ . 
Graphically construct $(\tilde{\eta}^{B}_s)_{s\geq 0}$ just like $(\eta^{B}_s)_{s\geq 0},$ using the same Poisson processes,  with the difference that all jumps
in the space-time region
\begin{align*}
\{(i,s)\in \Z\times \R_{+}:i\leq-t^{1/2+\varepsilon}, 0\leq s\leq t\}
\end{align*}
are suppressed. Denote by $\tilde{H}_{M-R}^{B}(t)$ the position of the $(M-R)$-th  hole (counted from left to right) of $\tilde{\eta}^{B}_{t}$.

Then, $\tilde{H}_{M-R}^{B}(t),\tilde{x}_{M}^{A}(t-t^{\kappa})$ are independent random variables.
Define the event
\begin{align*}
G_{t}=\{\tilde{H}_{M-R}^{B}(t)\neq H_{M-R}^{B}(t)\}\cup \{ \tilde{x}_{M}^{A}(t-t^{\kappa})\neq x_{M}^{A}(t-t^{\kappa})\}.
\end{align*} 
We show that 
\begin{align}\label{ja}
\lim_{t \to \infty}\Pb(G_{t})=0.
\end{align}
Let us first see how to finish the proof using \eqref{ja}: We have 
\begin{align*}
&\lim_{t\to \infty}\Pb(\min\{x_{M}^{A}(t-t^{\kappa})+(p-q)(t^{\kappa}+\tilde{C}t^{1/2}),x_{M}^{B}(t)\}\geq -R)
\\&=\lim_{t\to \infty}\Pb(\{\tilde{x}_{M}^{A}(t-t^{\kappa})+(p-q)(t^{\kappa}+\tilde{C}t^{1/2})\geq -R\}\cap \{ \tilde{H}_{M-R}^{B}(t)<-R\}\cap G_{t}^{c})
\\&= \lim_{t\to \infty}\Pb(\tilde{x}_{M}^{A}(t-t^{\kappa})+(p-q)(t^{\kappa}+\tilde{C}t^{1/2})\geq -R)\Pb(\tilde{H}_{M-R}^{B}(t)<-R)
\\&=F_{M,p}(C+\tilde{C})F_{M-R,p}(C),
\end{align*}
where for the last identity we used \eqref{need}.

Finally, consider the case $R\geq M$. Note that then $\Pb(x_{M}^{B}(t)\geq -R)=1$ and thus 
\begin{align*}&\lim_{t\to \infty}\Pb(\min\{x_{M}^{A}(t-t^{\kappa})+(p-q)(t^{\kappa}+\tilde{C}t^{1/2}),x_{M}^{B}(t)\}\geq -R)\\&=\lim_{t\to \infty}\Pb(x_{M}^{A}(t-t^{\kappa})+(p-q)(t^{\kappa}+\tilde{C}t^{1/2})\geq -R)\\&=F_{M,p}(C+\tilde{C}).
\end{align*}
This finishes the proof,  it remains to prove \eqref{ja}. To do this, 
note
\begin{align*}
\{\tilde{x}_{M}^{A}(t-t^{\kappa})\neq x_{M}^{A}(t-t^{\kappa})\}\subseteq \{\sup_{0\leq \ell\leq t-t^{\kappa}}x_{1}^{A}(\ell)\geq -(p-q)t^{\kappa}/4\}.
\end{align*}
By Theorem \ref{convthm},
\begin{align*}
\lim_{t \to \infty}\Pb (x_{1}^{A}(t-t^{\kappa})\geq -(p-q)t^{\kappa}/2)\leq \lim_{s\to -\infty}F_{1,p}(s)=0,
\end{align*}
so that 
\begin{align*}
&\lim_{t\to \infty}\Pb\left(\sup_{0\leq \ell\leq t-t^{\kappa}}x_{1}^{A}(\ell)\geq -(p-q)t^{\kappa}/4 \right)
\\&=\lim_{t\to \infty} \Pb\left (\sup_{0\leq \ell\leq t-t^{\kappa}}x_{1}^{A}(\ell)\geq -(p-q)t^{\kappa}/4,\,  x_{1}^{A}(t-t^{\kappa})\leq -(p-q)t^{\kappa}/2  \right).
\end{align*}
Start at time $0$ an ASEP from the initial data 
\begin{equation*}
\hat{\eta}_{0}(i)=\mathbf{1}_{\{i\geq -(p-q)t^{\kappa}/4\}}(i), \quad i\in \Z,
\end{equation*}
which is a shifted reversed step initial data. Denote by $\hat{x}_{0}(\ell)$ the position of the leftmost particle of $\hat{\eta}_{\ell}$  at time $\ell$.

Now on the event $\{\sup_{0\leq \ell\leq t-t^{\kappa}}x_{1}^{A}(\ell)\geq -(p-q)t^{\kappa}/4 \}$ there is a $\lambda_{1}\in [0,t-t^{\kappa}]$ such that
$x_{1}^{A}(\lambda_{1})\geq  -(p-q)t^{\kappa}/4$. 
Let us argue that we have 
\begin{equation}\label{llll}
x_{1}^{A}(\ell)\geq \hat{x}_{0}(\ell), \lambda_{1}\leq \ell\leq t-t^{\kappa}.
\end{equation}
To see this, note that if we start at time $\lambda_{1}$ an ASEP from $\mathbf{1}_{\Z_{\geq x_{1}^{A}(\lambda_{1})}}$, then for all times $\ell\geq \lambda_{1}$, the position of the leftmost particle of this ASEP is a lower bound for  $x_{1}^{A}(\ell)$. Furthermore, if  we start at time $\lambda_{1}$  an ASEP from $\hat{\eta}_{0}$ then for all  times $\ell\geq \lambda_{1}, $ the position of the leftmost particle  from this ASEP is a lower bound for the position of the leftmost particle of the ASEP started at time $\lambda_{1}$ from $\mathbf{1}_{\Z_{\geq x_{1}^{A}(\lambda_{1})}}$. Finally,  for all  times $\ell\geq \lambda_{1}, $  $\hat{x}_{0}(\ell)$ is   lower bound for  the position of the leftmost particle  of the ASEP started at time $\lambda_{1}$ from $\hat{\eta}_{0}.$ In particular,  for all  times $\ell\geq \lambda_{1}, $ $\hat{x}_{0}(\ell)$  is a lower bound for $x_{1}^{A}(\ell)$, i.e. \eqref{llll} holds.

In particular, we have thus shown 
\begin{equation*}
\{\sup_{0\leq \ell\leq t-t^{\kappa}}x_{1}^{A}(\ell)\geq -(p-q)t^{\kappa}/4\}\subseteq \{x_{1}^{A}(t-t^{\kappa})\geq \hat{x}_{0}(t-t^{\kappa})\}.
\end{equation*}
From this we may bound
\begin{align*}
& \lim_{t\to \infty} \Pb\left (\sup_{0\leq \ell\leq t-t^{\kappa}}x_{1}^{A}(\ell)\geq -(p-q)t^{\kappa}/4,\,  x_{1}^{A}(t-t^{\kappa})\leq -(p-q)t^{\kappa}/2 \right)
\\&\leq  \lim_{t\to \infty} \Pb\left ( \hat{x}_{0}(t-t^{\kappa})\leq -(p-q)t^{\kappa}/2 \right)
\\&=  \lim_{t\to \infty} \Pb\left ( x_{0}^{\mathrm{-step}}(t-t^{\kappa})\leq -(p-q)  t^{\kappa}/4 \right)
\\&=0,
\end{align*}
where in the last step we used Proposition \ref{block}. In total, we arrive at 
\begin{equation*}
\lim_{t \to \infty}\Pb(\tilde{x}_{M}^{A}(t-t^{\kappa})\neq x_{M}^{A}(t-t^{\kappa}))=0.
\end{equation*}
The proof of 
\begin{equation*}
\lim_{t \to \infty}\Pb(\tilde{H}_{M-R}^{B}(t)\neq H_{M-R}^{B}(t))=0
\end{equation*}
is almost identical, one notes 
\begin{align*}
\{(\tilde{H}_{M-R}^{B}(t)\neq H_{M-R}^{B}(t)\}\subseteq \{\inf_{0\leq \ell\leq t}H_{1}^{B}(\ell)\leq -t^{1/2+\varepsilon}\}
\end{align*}
and deduces  $\lim_{t\to \infty}\Pb(\{\inf_{0\leq \ell\leq t}H_{1}^{B}(\ell)\leq -t^{1/2+\varepsilon}\})=0$ from 
\begin{align*}
\lim_{t \to \infty}\Pb (H_{1}^{B}(t)\leq -t^{1/2+\varepsilon}/2)=0.
\end{align*}

So we have shown  \eqref{ja}. 
\end{proof}

\section{Proof of upper bound}\label{upsec}
Recall that we had defined in \eqref{yn} the minimum
\begin{equation}\label{y2}
y_{n}(t)=\min\{x_{n}^{A}(t),x_{n}^{B}(t)\},
\end{equation}
which under the basic coupling of $(\eta_{t},\eta^{A}_{t},\eta^{B}_{t},t \geq 0)$ satisfies $y_{n}(t)\geq x_{n}(t)$.
 The following theorem gives the \textit{discrete} $t\to\infty$  limit law for $y_{M+\lambda M^{1/3}}(t)$ inside the shock  as well as for $x_{M+\lambda M^{1/3}}(t)$ to the left of the shock.
\begin{tthm}\label{main}Let  $(\eta_{t},\eta^{A}_{t},\eta^{B}_{t},t \geq 0)$ be coupled via the basic coupling.
 For  fixed $C \in \R,$ consider  the minimum \eqref{y2} and the ASEP with initial data \eqref{IC}. 
Let $R \in \Z,M\geq 1,\lambda\in\R$. Then  
\begin{equation}\label{prodform}
\lim_{t \to \infty}\Pb(y_{M+\lambda M^{1/3}}(t)\geq -R)=\begin{cases} F_{M+\lambda M^{1/3},p}(C)&\mathrm{for} \,  R \geq M+\lambda M^{1/3} \\
F_{M+\lambda M^{1/3},p}(C)F_{M+\lambda M^{1/3}-R,p}(C) &\mathrm{for} \,  R < M+\lambda M^{1/3}.
\end{cases}
\end{equation}
Furthermore, for $s\in \R\setminus\{0\}$
\begin{equation}\label{cutoff}
\lim_{t \to \infty}\Pb(x_{M+\lambda M^{1/3}}(t)\geq -(p-q)s t^{1/2})=
F_{M+\lambda M^{1/3},p}(s+C)\mathbf{1}_{\{s>0\}}.
\end{equation}

\end{tthm}
Using Proposition \ref{FGUE}, we arrive at continuous limit distributions by sending $M\to\infty$:
\begin{cor}\label{GUEGUEcor}
Consider ASEP with the initial data \eqref{IC} and $C=C(M)$ as in \eqref{C}.
Then 
\begin{equation}\label{dududu}
\lim_{M \to \infty} \lim_{t \to \infty}\Pb\left( y_{M+\lambda M^{1/3}}(t)\geq -\xi M^{1/3}\right)=F_{\GUE}(-\lambda)F_{\GUE}(\xi-\lambda),
\end{equation}
for $\lambda,\xi \in \R$. Furthermore,  we have for $s \in \R \setminus\{0\}$
\begin{equation}\label{dududu2}
\lim_{M \to \infty} \lim_{t \to \infty}\Pb\left( x_{M+\lambda M^{1/3}}(t)\geq -s\sqrt{p-q}t^{1/2}M^{-1/6}\right)=F_{\GUE}(s-\lambda)\mathbf{1}_{\{s>0\}}.
\end{equation}
\end{cor}
\begin{proof}
Note that  $C$ is  as in \eqref{C}. Then the result follows from Theorem \ref{main} and  Proposition \ref{FGUE} together with a simple change of variable.
\end{proof}

We split the proof of Theorem \ref{main} in two parts.
\begin{proof}[Proof of \eqref{prodform}]
We assume   $\lambda=0$ w.l.o.g.
We define the event 
\begin{equation*}
A_{t}=\{ \left| x_{M}^{A}(t)-x_{M}^{A}(t-t^{\kappa})-(p-q)t^{\kappa}\right| \leq \varepsilon t^{1/2}\}.
\end{equation*}
We easily see the relations 
\begin{align*}
\{y_{M}(t)\geq -R\}\cap A_{t}&=\{\min\{x_{M}^{A}(t)-x_{M}^{A}(t-t^{\kappa})-(p-q)t^{\kappa}+x_{M}^{A}(t-t^{\kappa})+(p-q)t^{\kappa},x_{M}^{B}(t)\}\geq -R\}\\&\quad\,\cap A_{t}
\subseteq\{\min\{\varepsilon t^{1/2}+x_{M}^{A}(t-t^{\kappa})+(p-q)t^{\kappa},x_{M}^{B}(t)\}\geq -R\}
\end{align*}
and likewise
\begin{align*}
\{y_{M}(t)\geq -R\}\cap A_{t} \supseteq\{\min\{-\varepsilon t^{1/2}+x_{M}^{A}(t-t^{\kappa})+(p-q)t^{\kappa},x_{M}^{B}(t)\}\geq -R\}\cap A_{t}.
\end{align*}
Thus we have 
\begin{equation}
\begin{aligned}\label{fin1}
\lim_{t\to \infty} \Pb (y_{M}(t)\geq -R)\leq   \lim_{t\to \infty} \Pb(\min\{\varepsilon t^{1/2}+x_{M}^{A}(t-t^{\kappa})+(p-q)t^{\kappa},x_{M}^{B}(t)\}\geq -R)+\Pb(A_{t}^{c})
\end{aligned}
\end{equation}
and 
\begin{equation}\label{fin2}
\begin{aligned}
\lim_{t\to \infty} \Pb (y_{M}(t)\geq -R)
 \geq  \lim_{t\to \infty} \Pb(\{\min\{-\varepsilon t^{1/2}+x_{M}^{A}(t-t^{\kappa})+(p-q)t^{\kappa},x_{M}^{B}(t)\}\geq -R\})-\Pb(A_{t}^{c}).
\end{aligned}
\end{equation}

Applying Propositions   \ref{prop0} and  \ref{indepprop} to the inequalities \eqref{fin1},\eqref{fin2} yields for $R<M$
\begin{equation*}
F_{M,p}(C-\varepsilon/(p-q))F_{M-R,p}(C)\leq \lim_{t\to \infty} \Pb (y_{M}(t)\geq -R)\leq F_{M,p}(C+\varepsilon/(p-q))F_{M-R,p}(C),
\end{equation*}
and for $R\geq M$
\begin{equation*}
F_{M,p}(C-\varepsilon/(p-q))\leq \lim_{t\to \infty} \Pb (y_{M}(t)\geq -R)\leq F_{M,p}(C+\varepsilon/(p-q))
\end{equation*}
finishing the proof since $\varepsilon>0$ is arbitrary.
\end{proof}

\begin{proof}[Proof of \eqref{cutoff}]
 To lighten the notation, we may set w.l.o.g.
\begin{equation*}
\lambda=0
\end{equation*}
in this proof.
For $s<0$, \eqref{cutoff} easily follows from  \eqref{prodform} by sending $R\to-\infty$.

Let us  prove \eqref{cutoff} for $s>0$. Denote by $x_{-n}^{-\mathrm{step}}(0)=n,n\geq 0$ the reversed step initial data, and  couple $(x_{-n}^{-\mathrm{step}}(\ell))_{\ell\geq 0, n\geq 0}$ with $x_{M}(t),x_{M}^{A}(t),x_{M}^{B}(t)$ with the basic coupling. 
By  Theorem \ref{convthm} and since $x_{M}(t)\leq x_{M}^{A}(t)$ we have
\begin{align*}
F_{M,p}(s+C)=&\lim_{t \to \infty}\Pb(x_{M}^{A}(t)\geq  -(p-q)st^{1/2},x_{M}(t)< -(p-q)st^{1/2})
\\&+\lim_{t \to \infty}\Pb(x_{M}^{A}(t)\geq  -(p-q)st^{1/2},x_{M}(t)\geq  -(p-q)st^{1/2})
\\& =\lim_{t \to \infty}\Pb(x_{M}^{A}(t)\geq  -(p-q)st^{1/2},x_{M}(t)< -(p-q)st^{1/2})
\\&+\lim_{t \to \infty}\Pb(x_{M}(t)\geq  -(p-q)st^{1/2}).
\end{align*} 
It thus suffices to prove
\begin{equation*}
\lim_{t \to \infty}\Pb(x_{M}(t)< -(p-q)st^{1/2}, x_{M}^{A}(t)\geq  -(p-q)st^{1/2})=0.
\end{equation*} 
The reason why this is true is that in order for $x_{M}(t)\neq x_{M}^{A}(t)$ to hold, $x_{M}(t)$ must have "felt" the particle $x_{0}$. But by Proposition \ref{block}, $x_{0}$ does not go to the left of $-t^{\delta}$ for $\delta>0$ small, so if  $x_{M}(t)$ has felt the presence of $x_{0},$ $x_{M}(t)< -(p-q)st^{1/2}$   cannot hold. To make this precise, 
 define the stopping times 
\begin{equation}
\begin{aligned}\label{tau}
&\tau_{0}=0
\\&\tau_{i}=\inf\{\ell: x_{i}(\ell)\neq x_{i}^{A}(\ell)\}, \, i \geq 1.
\end{aligned} 
\end{equation}

We show
\begin{align}\label{omit}
\{x_{M}^{A}(t)\neq x_{M}(t)\}\subseteq \mathcal{B}_{t}:=\{ 0=\tau_{0}<\tau_{1}<\cdots<\tau_{M}\leq t, x_{i-1}(\tau_{i}) -x_{i}(\tau_{i})=1,i=1,\ldots,M\}.
\end{align}
To see \eqref{omit},  note $0<\tau_{M}\leq t$ on $\{x_{M}(t)\neq x_{M}^{A}(t)\}.$ Recall further $x_{M}(\ell)\leq x_{M}^{A}(\ell)$ for all $\ell\geq 0$. Then we have \begin{equation*}x_{M}(\tau_{M})\neq x_{M}^{A}(\tau_{M}),x_{M}(\tau_{M}^{-})= x_{M}^{A}(\tau_{M}^{-}).\end{equation*}
Now $x_{M}(\tau_{M}^{-})= x_{M}^{A}(\tau_{M}^{-})$ implies  $x_{M+1}(\tau_{M}^{-})= x_{M+1}^{A}(\tau_{M}^{-})$ : Assume to the contrary $x_{M}(\tau_{M}^{-})= x_{M}^{A}(\tau_{M}^{-}), x_{M+1}(\tau_{M}^{-})\neq  x_{M+1}^{A}(\tau_{M}^{-})$ both hold.  Since  $x_{M+1}(\tau_{M}^{-})\neq  x_{M+1}^{A}(\tau_{M}^{-})$ is equivalent to 
$x_{M+1}(\tau_{M}^{-})<  x_{M+1}^{A}(\tau_{M}^{-}),$ we have that $x_{M+1}(\tau_{M}^{-})\neq  x_{M+1}^{A}(\tau_{M}^{-})$ implies  \begin{equation*}x_{M+1}(\tau_{M}^{-})<  x_{M+1}^{A}(\tau_{M}^{-})< x_{M}^{A}(\tau_{M}^{-})= x_{M}(\tau_{M}^{-}).\end{equation*}But this cannot happen since 
then $\eta_{\tau_{M}^{-}}^{A}(x_{M+1}^{A}(\tau_{M}^{-}))=1>\eta_{\tau_{M}^{-}}(x_{M+1}^{A}(\tau_{M}^{-}))$ in contradiction to $\eta_{t}^{A}\leq \eta_{t}$ for all $t$.

Now the fact  that $x_{M}(\tau_{M}^{-})= x_{M}^{A}(\tau_{M}^{-}), x_{M+1}(\tau_{M}^{-})= x_{M+1}^{A}(\tau_{M}^{-})$  hold implies that the only way the discrepancy  $x_{M}(\tau_{M})\neq  x_{M}^{A}(\tau_{M})$
can be created is by a jump to the right of $x_{M}^{A}$ that $x_{M}$ does not make (the other possibility to create this discrepancy would be by a jump of $x_{M}$ to the left that $x_{M}^{A}$ does not make, but since 
$x_{M+1}(\tau_{M}^{-})= x_{M+1}^{A}(\tau_{M}^{-}),$ $x_{M}$ and $x_{M}^{A}$ can only jump together to the left at time $\tau_{M}$). This shows that at time $\tau_{M}$ a jump of $x_{M}$ has been suppressed by the presence of $x_{M-1}$ and it also shows that  $x_{M-1}(\tau_{M})<x_{M-1}^{A}(\tau_{M}),$ which in turn implies that 
\begin{equation*}
0<\tau_{M-1}<\tau_{M}.
\end{equation*}
 Repeating the preceding argument, we see that at time $\tau_{M-1}$ a jump of $x_{M-1}$ was suppressed by the presence of $x_{M-2}.$ Iteratively, we obtain $0<\tau_{1}< \cdots<\tau_{M} \leq t $ and  that at time $\tau_{i}$, a jump of $x_{i}$ is suppressed by the presence of $x_{i-1}$, $i=1,\ldots,M$. In particular, \eqref{omit} holds and in fact we have
 \begin{equation}\label{BBB}
 \mathcal{B}_{t}= \{\tau_{M}\leq t\}.
 \end{equation}

We can bound 
\begin{align}\nonumber
&\lim_{t \to \infty}\Pb(x_{M}(t)< -(p-q)st^{1/2}, x_{M}^{A}(t)\geq  -(p-q)st^{1/2})
\\&=\lim_{t \to \infty}\Pb(x_{M}(t)< -(p-q)st^{1/2}, x_{M}^{A}(t)\geq  -(p-q)st^{1/2},x_{M}(t)\neq x_{M}^{A}(t))\nonumber
\\&\leq \lim_{t \to \infty}\Pb(\{x_{M}(t)< -(p-q)st^{1/2}\}\cap \mathcal{B}_{t} )\label{CMP}.
\end{align} 
and wish to prove   that \eqref{CMP} equals zero.  

Define the event 
\begin{equation*}
E_{i}=\{\inf_{\tau_{i}\leq \ell \leq t}x_{i}(\ell) \leq -(i+1)t^{\delta/2}\}\cap\mathcal{B}_{t}.
\end{equation*}
By \eqref{bms},
\begin{equation}\label{E0}
\Pb(E_{0})\leq \Pb(\inf_{0\leq \ell\leq t}x_{0}^{-\mathrm{step}}(\ell)\leq -t^{\delta/2})\leq C_{1}e^{-C_{2}t^{\delta/2}}.
\end{equation}
Using $\Pb(E_{i})=\Pb((E_{i}\cap E_{i-1})\cup ( E_{i}\cap E_{i-1}^{c}),$ we may bound for  $i\geq 1$ 
\begin{align}\label{bbb}
\Pb(E_{i})&\leq \Pb(E_{i-1}) +   \Pb\left(\{\inf_{\tau_{i-1}\leq \ell\leq t}x_{i-1}(\ell)> -i t^{\delta/2}\}\cap \{\inf_{\tau_{i}\leq \ell\leq t}x_{i}(\ell)\leq  -(i+1) t^{\delta/2}\}\cap \mathcal{B}_{t}\right)
\\&\leq \Pb(E_{i-1}) +   \Pb\left(\{x_{i}(\tau_{i})\geq  -i t^{\delta/2}\}\cap \{\inf_{\tau_{i}\leq \ell\leq t}x_{i}(\ell)\leq  -(i+1) t^{\delta/2}\}\cap \mathcal{B}_{t}\right)
\end{align}
where we used $x_{i-1}(\tau_{i}) -x_{i}(\tau_{i})=1$ on $\mathcal{B}_{t}$.
Start an ASEP at time $0$ from the initial data
\begin{equation*}
\eta^{i}_{0}(j)=\mathbf{1}_{\{j\geq -i t^{\delta/2}\}}(j),\quad j \in \Z,
\end{equation*}
and denote by $x_{0}^{i}(s)$ the position of the leftmost particle of $\eta^{i}_{s}$ at time $s$.
Note 
\begin{equation*}
\{x_{i}(\tau_{i})\geq  -i t^{\delta/2}\} \cap \mathcal{B}_{t}\subseteq \{x_{0}^{i}(\ell )\leq x_{i}(\ell), \tau_{i}\leq \ell \leq t\}\cap \mathcal{B}_{t}.
\end{equation*}
Consequently,
we may bound
\begin{equation*}
\Pb(E_{i})\leq \Pb(E_{i-1}) +\Pb\left(\inf_{0\leq \ell \leq t}x_{0}^{i}(\ell )\leq -(i+1)t^{\delta/2}\right)\leq  \Pb(E_{i-1})+C_{1}e^{-C_{2}t^{\delta/2}},
\end{equation*}
implying by \eqref{E0} that we may  bound 
\begin{equation}\label{EM}
\Pb(E_{M})\leq (M+1)C_{1}e^{-C_{2}t^{\delta/2}}
\end{equation}
for $t$ sufficiently large.
Next we make the observation  that since $(M+1)t^{\delta/2}\leq s(p-q)t^{1/2}$ 
\begin{equation*}
\{x_{M}(t)< -(p-q)st^{1/2}\}\cap \mathcal{B}_{t}\setminus E_{M}=\emptyset.
\end{equation*}
This implies 
\begin{equation*}
\Pb(\{x_{M}(t)< -(p-q)st^{1/2}\}\cap \mathcal{B}_{t} )\leq \Pb(E_{M})\leq (M+1)C_{1}e^{-C_{2}t^{\delta/2}},
\end{equation*}
finishing the proof of \eqref{cutoff} for $s>0$.
\end{proof}

\section{Proof of lower bound}\label{lowsec}
Here we provide the lower bound for the double limit \eqref{Jamie}, see Theorem \ref{lowerthm} below.
In this Section, to make the needed adaptions to prove Theorem \ref{GUEGUE2} as easy as possible, we carry around with us   the parameter 
\begin{equation*}
\nu\in[0,3/7).
\end{equation*}
To prove Theorem \ref{GUEGUE} (in which $\nu$ does not appear),   we may set  $\nu=0$ wherever it appears in this Section. 
As it was already sketched in Section \ref{method}, we wish to first show  that already at a time point $t-t^{\chi}<t,$ the particles  $x_{0}(t-t^{\chi})$  and $x_{M+\lambda M^{1/3}}(t-t^{\chi})$ have reached certain positions with a probability that is asymptotically bounded from 
below by $F_{\GUE}(-\lambda)F_{\GUE}(\xi-\lambda)
$. 
To show this, we first consider $x_{0}(t-t^{\chi}).$ 
\begin{prop}\label{X0}Consider ASEP with the initial data \eqref{IC} and $C=C(M)$ as in \eqref{C}. 
Let $\delta\in (0, 1/2-7\nu /6)$ and $\chi \in  (\nu+\delta,1/2-\nu/6).$ 
Then 
\begin{equation}\label{prop34}
\lim_{M \to \infty}\lim_{t\to \infty}\Pb(x_{0}(t-t^{\chi})>M-M^{1/3}\xi)\geq F_{\GUE}(\xi).
\end{equation}
\end{prop}
To prove Proposition \ref{X0}, we use the general strategy outlined in  Section \ref{method}, the only difference is that
here we wish to lower bound the position of $x_{0}$ at time $t-t^{\chi}$, hence we will construct an event $\tilde{\mathcal{E}}_{t-2t^{\chi}}$ which depends only on what happens during $[0,t-2t^{\chi}].$ 
The relation \eqref{allgemeine} from  Section \ref{method}  is \eqref{subseteq} here. 
Also note that by comparing  $x_{0}(t-t^{\chi})\leq x_{0}^{B}(t-t^{\chi})$ it is easy to see that the inequality \eqref{prop34} holds in the other direction, showing that \eqref{prop34} is in fact an identity.
\begin{proof}[Proof of Proposition \ref{X0}]
We define the initial data 
\begin{equation}\label{D}
x_{n}^{D}(0)=-n, \quad -\lfloor (p-q) (t-C(M)t^{1/2})\rfloor\leq n\leq 0
\end{equation}
(we have avoided the denomination $x_{n}^{C}$ here to reserve the letter $C$ for constants). We denote by $\eta^{D}_{s},0\leq s\leq t,$  the ASEP started from \eqref{D}.
It is easy to deduce from  $\eta_{0}^{D}(j)\leq \eta_{0}(j),j\leq 0,$ and 
$\eta_{0}^{D}(j)\geq \eta_{0}(j),j\geq 0,$
 that under the basic coupling we have the inequality 
\begin{equation*}
x_{0}^{D}(s)\leq x_{0}(s), \quad s\geq 0.
\end{equation*}
Thus it suffices to prove \eqref{prop34} for $x_{0}^{D}(t-t^{\chi})$.
We now go through the steps of the general strategy of Section \ref{method}. We stress that all processes  appearing in the proof are coupled via the basic coupling.

\textit{1.Step: Establishing the relation \eqref{allgemeine}:}\\

\begin{figure}\begin{center}
\begin{tikzpicture}[scale=1.6]

  \draw[thick, ->] (-1,0) -- (5.7,0);

   \draw (-1.8,1) node[below] {\Large{$\eta_{t-2t^{\chi}}^{D}:$}};
      \draw (1.1,-0.1) node[below] {$x_{0}^{D}(t-2t^{\chi})$};
    \draw (-0.7,0.2) circle (0.1); 
\draw (3.05,0.2) circle (0.1);
     \draw (3,-0.1) node[below] {$H_{Z}^{D}(t-2t^{\chi})$};
   \draw (3.8,-0.1) node[below] {$t^{\delta+\nu}$};
\draw[thick] (0.3,-0.1) node[below] {$-2M$};
   \foreach \x in {3.05,3.8,1.1,0.3}
      \draw[very thick] (\x,0.075)--(\x,-0.075);
 
  \foreach \x in {-0.2,-0.45,0.05,-0.95}
      \draw (\x,0.2) circle(0.1);
   \filldraw (1.1,0.2) circle (0.1);
   \foreach \x in {-0.2,-0.45,0.05,-0.95}
      \draw (\x,0.2) circle(0.1);

      \begin{scope}[yshift=-1.2cm]
      
         \draw (-1.8,1) node[below] {\Large{$\hat{\eta}_{t-2t^{\chi}}^{D}:$}};
  \draw[thick, ->] (-1,0) -- (5.7,0);

  \filldraw (1.1,0.2) circle (0.1);
      \draw (1.1,-0.1) node[below] {$\hat{x}_{0}^{D}(t-2t^{\chi})$};

\draw (3.05,0.2) circle (0.1);
     \draw (3,-0.1) node[below] {$H_{Z}^{D}(t-2t^{\chi})$};
   \draw (3.8,-0.1) node[below] {$t^{\delta+\nu}$};
\draw[thick] (0.3,-0.1) node[below] {$-2M$};

   \foreach \x in {3.05,3.8,1.1,0.3}
      \draw[very thick] (\x,0.075)--(\x,-0.075);
 
  \foreach \x in {-0.2,-0.7, -0.45,0.05,-0.95}
      \draw (\x,0.2) circle(0.1);
 
   \foreach \x in {-0.2,-0.45,0.05,-0.95}
      \draw (\x,0.2) circle(0.1);
      
   \foreach \x in {3.3,3.55,3.8,4.05,4.3,4.55,4.8,5.05,5.3,5.55}
      \filldraw (\x,0.2) circle(0.1);
 \end{scope}

    \begin{scope}[yshift=-2.4cm]
              \draw (-1.8,1) node[below] {\Large{$\bar{\eta}_{t-2t^{\chi}}:$}};
      
  \draw[thick, ->] (-1,0) -- (5.7,0);

\draw (3.3,0.2) circle (0.1);
     \draw (2.55,-0.1) node[below] {$ t^{\delta+\nu}-Z-2M$};
   \draw (3.8,-0.1) node[below] {$t^{\delta+\nu}$};
\draw (0.3,-0.1) node[below] {$-2M$};
\filldraw (0.3,0.2) circle(0.1);

   \foreach \x in {2.55,3.8,0.3}
      \draw[very thick] (\x,0.075)--(\x,-0.075);
 
  \foreach \x in {-0.2,-0.7, -0.45,0.05,-0.95,3.8}
      \draw (\x,0.2) circle(0.1);
 
   \foreach \x in {-0.2,-0.45,0.05,-0.95,2.8,3.05,3.3,3.55}
      \draw (\x,0.2) circle(0.1);
      
   \foreach \x in {0.55,0.8,1.05, 1.3,1.55,1.8,2.05,2.3,2.55,4.05,4.3,4.55,4.8,5.05,5.3,5.55}
      \filldraw (\x,0.2) circle(0.1);
 \end{scope}
\end{tikzpicture}\end{center}
\caption{From top to bottom: The three particle configurations $\eta_{t-2t^{\chi}}^{D}, \hat{\eta}_{t-2t^{\chi}}^{D},\bar{\eta}_{t-2t^{\chi}}$ on the event $\tilde{\mathcal{E}}_{t-2t^{\chi}}$ defined in \eqref{tildeE}. Holes/particles are shown as white/black circles.   $\hat{\eta}_{t-2t^{\chi}}^{D}$ is obtained from 
$\eta_{t-2t^{\chi}}^{D}$ by replacing all holes to the right of   $H_{Z}^{D}(t-2t^{\chi})$ by particles.
For  $\ell\geq t-2t^{\chi}$, the position $\hat{x}_{0}^{D}(\ell)$ of the  leftmost particle of $\hat{\eta}_{\ell}^{D}$ is a lower bound for the position of the leftmost particle  $x_{0}^{D}(\ell)$ of $\eta^{D}_{\ell}$. Equally, the position $\bar{x}_{0}(\ell)$ of the leftmost particle of $\bar{\eta}_{\ell}$ is a lower bound for $\hat{x}_{0}^{D}(\ell)$. 
In particular, on the event $\tilde{\mathcal{E}}_{t-2t^{\chi}},$ the particle position $x_{0}^{D}(t-t^{\chi})$ is bounded from below by $\bar{x}_{0}(t-t^{\chi}),$ and we bound $\bar{x}_{0}(t-t^{\chi})$ via Proposition \ref{DOIT}. }  
\label{Graph}
\end{figure}
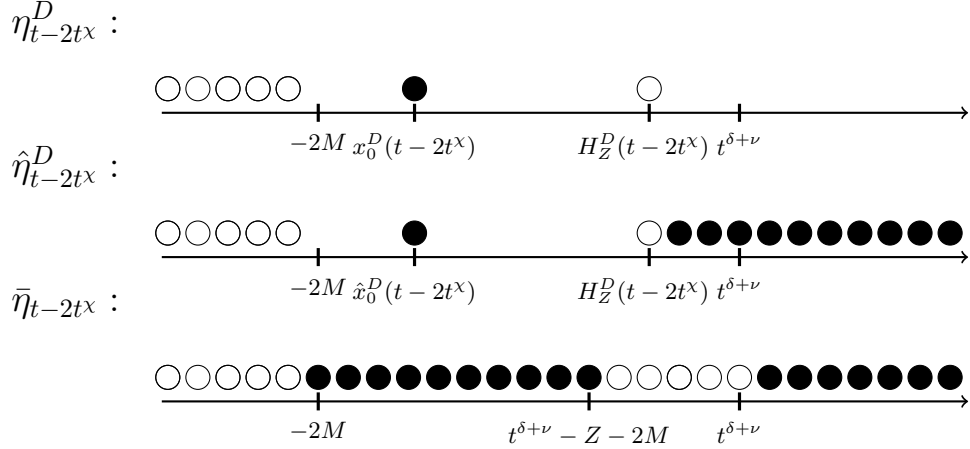
We label the holes of $ \eta_{0}^{D}$ as 
\begin{align*}
H_{n}^{D}(0)=\begin{cases}
n+\lfloor (p-q) (t-C(M)t^{1/2})\rfloor,\quad &n\geq 1\\
n-1, & n\leq 0.
\end{cases}
\end{align*}
We define the integer
\begin{equation}\label{Z}
 Z=\lfloor M-\xi M^{1/3}+M^{1/4}\rfloor
 \end{equation}
 and the deterministic  particle configuration $\bar{\eta}_{t-2t^{\chi}}$ (see Figure \ref{Graph})
 \begin{align*}
\bar{\eta}_{t-2t^{\chi}}(j):=\mathbf{1}_{\{-2M,\ldots,\lfloor t^{\delta+\nu}\rfloor-Z-2M\}}(j)+\mathbf{1}_{\Z_{\geq \lfloor t^{\delta+\nu}\rfloor+1}}(j)=\eta^{-2M,\lfloor t^{\delta+\nu}\rfloor-Z-2M,\lfloor t^{\delta+\nu}\rfloor}(j),j\in \Z ,
\end{align*} 

where the notation    $\eta^{a,b,N},a\leq b\leq N,$ was  introduced in  \eqref{abN}.

At time $t-2t^{\chi},$ we start an ASEP  from  $\bar{\eta}_{t-2t^{\chi}}$ 
and we denote by $\bar{x}_{0}(t-t^{\chi})$ the position of the leftmost particle of $ \bar{\eta}_{t-t^{\chi}}.$ 
Following the general strategy, we define
\begin{equation}\label{tildeE}
\tilde{\mathcal{E}}_{t-2t^{\chi}}=\{x_{0}^{D}(t-2t^{\chi})\geq -2M\}\cap\{H_{Z}^{D}(t-2t^{\chi})\leq \lfloor t^{\delta+\nu}\rfloor\},\end{equation}

and  the relation \eqref{allgemeine} from the general strategy that we have to prove is
 \begin{equation}\label{subseteq}
\tilde{\mathcal{E}}_{t-2t^{\chi}}\subseteq \{  \bar{x}_{0}(t-t^{\chi})\leq x_{0}^{D}(t-t^{\chi})\}.
 \end{equation}
To prove \eqref{subseteq}, we define an  auxiliary ASEP which starts at time $t-2t^{\chi}$ from 

\begin{align*}
\hat{\eta}^{D}_{t-2t^{\chi}}(j)=\begin{cases}
1 \quad & j>H_{Z}^{D}(t-2t^{\chi})
\\ \eta^{D}_{t-2 t^{\chi}}(j), \quad & j\leq  H_{Z}^{D}(t-2t^{\chi}),
\end{cases}
\end{align*} 
see Figure \ref{Graph}.
Denote by  $\hat{x}_{0}^{D}(t-t^{\chi})$ the position of the leftmost particle of $\hat{\eta}^{D}_{t-t^{\chi}}$. Again, under the basic coupling we have
\begin{equation}
\hat{x}_{0}^{D}(t-t^{\chi})\leq x_{0}^{D}(t-t^{\chi}).
\end{equation}
Recalling the partial order $\preceq$ from \eqref{preceq},  let us compute that  
 \begin{equation}\label{subseteq2}
\tilde{\mathcal{E}}_{t-2t^{\chi}}\subseteq \{ \bar{\eta}_{t-t^{\chi}} \preceq \hat{\eta}^{D}_{t-t^{\chi}} \}.
 \end{equation}
 Since the partial order $\preceq$ is preserved as time evolves,  it suffices to show that on $\tilde{\mathcal{E}}_{t-2t^{\chi}},$ we have
 $\bar{\eta}_{t-2t^{\chi}}\preceq \hat{\eta}^{D}_{t-2t^{\chi}}.$
To see this, note that on $\tilde{\mathcal{E}}_{t-2t^{\chi}},$ we have that $\hat{\eta}^{D}_{t-2t^{\chi}}(j)=\bar{\eta}_{t-2t^{\chi}}(j)$ for $j \notin \{-2M,\ldots, \lfloor t^{\delta+\nu}\rfloor\}$ . Hence, to show \eqref{subseteq2}, it suffices to check that  on $\tilde{\mathcal{E}}_{t-2t^{\chi}},$ for $j \in \{-2M,\ldots, \lfloor t^{\delta+\nu}\rfloor\}$
we have 
\begin{equation}\label{68}
\sum_{r=j}^{\infty}1-\hat{\eta}^{D}_{t-2t^{\chi}}(r)  \leq \sum_{r=j}^{\infty}1-\bar{\eta}_{t-2t^{\chi}}(r).
\end{equation}
To check \eqref{68}, note that on the event  $\tilde{\mathcal{E}}_{t-2t^{\chi}}$ we can compute
 \begin{equation}
 \sum_{r=\lfloor t^{\delta+\nu}\rfloor-Z-2M+1}^{\infty}1-\bar{\eta}_{t-2t^{\chi}}(r)
=Z+2M -1=  \sum_{r=-2M}^{\infty}1-\hat{\eta}^{D}_{t-2t^{\chi}}(r),\end{equation}
establishing \eqref{68}. Applying Lemma \ref{lem} to \eqref{subseteq2}, we see that  \eqref{subseteq} holds.


\textit{2. Step: Bounding $\Pb(\tilde{\mathcal{E}}_{t-2t^{\chi}})$:}\\
Following the general strategy, we show next 
\begin{equation}\label{alpha}
\lim_{M \to \infty}\lim_{t\to \infty}\Pb(\tilde{\mathcal{E}}_{t-2t^{\chi}})\geq F_{\GUE}(\xi).
\end{equation}

Recall the collection of holes $H_{n}^{B}(0)$  from \eqref{holes}.
Define the event
\begin{equation*}
\mathcal{D}_{t-2t^{\chi}}=\{H_{Z}^{B}(t-2t^{\chi}) <Z-1\}.
\end{equation*}    
It follows from Theorem \ref{convthm} and $\chi<1/2-\nu/6\leq 1/2$ that $\lim_{t\to \infty}\Pb(\mathcal{D}_{t-2t^{\chi}})=F_{Z,p}(C(M))$ and hence by Proposition \ref{FGUE}
 \begin{equation}\label{kolle}
 \lim_{M\to \infty}\lim_{t\to \infty}\Pb(\mathcal{D}_{t-2t^{\chi}})=F_{\GUE}(\xi).
 \end{equation}
 
We define (in direct analogy to \eqref{tau})
\begin{equation*}
\tau_{Z}^{D}=\inf\{\ell: H_{Z}^{D}(\ell)\neq H_{Z}^{B}(\ell)\}.
\end{equation*}
Since $H_{Z}^{D}(s)\geq Z-1$ for all $s\geq 0$  we have
\begin{equation}\label{label1}
\mathcal{D}_{t-2t^{\chi}}\subseteq \{\tau_{Z}^{D}\leq t-2t^{\chi}\}.
\end{equation}
Next we note that there are constants $C_{1},C_{2}>0$ so that for  $t$ sufficiently large 
\begin{equation}\begin{aligned}\label{label2}
&\Pb(\{\tau_{Z}^{D}\leq t-2t^{\chi}\}\cap \{H_{Z}^{D}(t-2t^{\chi})>t^{\delta+\nu}\})
\\&\leq \Pb(\{\tau_{Z}^{D}\leq t-2t^{\chi}\}\cap \{H_{Z}^{D}(t-2t^{\chi})>(Z+1)t^{\delta/2}\})\leq (Z+1)C_{1}e^{-C_{2}t^{\delta/2}}.
\end{aligned}
\end{equation}
The proof of \eqref{label2} is directly analogue to that of \eqref{EM}, one simply has to replace the role of the $x_{i}$ by $H_{i}^{D}$.
 From \eqref{kolle} and \eqref{label2} we deduce
  \begin{equation}\label{leqgue}
 \lim_{M\to \infty}\lim_{t\to \infty}\Pb(H_{Z}^{D}(t-2t^{\chi})\leq t^{\delta+\nu})\geq F_{\GUE}(\xi).
 \end{equation}

  Furthermore, 
\begin{align*}
\lim_{t\to\infty}\Pb\left( x_{0}^{D}(t-t^{\chi})<-2M\right)\leq \lim_{t \to\infty}\Pb\left( x_{0}^{-\mathrm{step}}(t-t^{\chi})<-2M\right)\leq C_{1}e^{-C_{2}M}
\end{align*}
using \eqref{bms2}, which together with \eqref{leqgue} proves  \eqref{alpha}.

 \textit{3. Step: Proving the relation \eqref{such that}:}\\
As last step from the general strategy, we need to show \eqref{such that}.  Recalling $\bar{\eta}_{t-2t^{\chi}}=\eta^{-2M,\lfloor t^{\delta+\nu}\rfloor-Z-2M,\lfloor t^{\delta+\nu}\rfloor}$, by Proposition \ref{DOIT} (with $R=M^{1/4}, \mathcal{M}=t^{\delta+\nu}-Z+1<t^{\chi},\varepsilon=1$), we have 
 \begin{equation}\label{ineqq2}
  \Pb\left(\{\bar{x}_{0}(t-t^{\chi})\geq M-\xi M^{1/3}\}   \right)\geq 1-C_{1}e^{-C_{2}M^{1/4}}-1/(t^{\delta+\nu}-Z+1).
 \end{equation}

Since  $\bar{x}_{0}(t-t^{\chi}),\tilde{\mathcal{E}}_{t-2t^{\chi}}$   are independent by construction, we may bound as in \eqref{such that}
 \begin{align}
& \notag \Pb\left(x_{0}^{D}(t-t^{\chi})\geq M-\xi M^{1/3}\right)\geq \Pb\left(\bar{x}_{0}(t-t^{\chi})\geq M-\xi M^{1/3} \right) \Pb(\tilde{\mathcal{E}}_{t-2t^{\chi}}),
 \end{align}
 finishing  the proof by \eqref{ineqq2} and \eqref{alpha}.

Let us note that   one of the  reasons why we assumed $\nu<3/7$ is to obtain  the inequality  \eqref{ineqq2}. If we had  $\nu\geq 1/2-\nu/6,$ (i.e. $\nu\geq 3/7 $), the ASEP started from $\bar{\eta}_{t-2t^{\chi}}$  could not come close to  equilibrium (specifically, hit the reversed step initial data, see the proof of Proposition \ref{DOIT})
during $[t-2t^{\chi},t-t^{\chi}]$  because $\chi<1/2-\nu/6$ (and we cannot increase $\chi$  to be bigger than $1/2-\nu/6$ without destroying the convergence of $\Pb(\mathcal{D}_{t-2t^{\chi}})$).
Without the mixing of $\bar{\eta}_{t-2t^{\chi}}$  though, we do  not get the needed inequality  \eqref{ineqq2}.

  \end{proof}

Finally, we can now provide the lower bound for the double limit \eqref{Jamie}. For this, we follow again the general strategy outlined in Section \ref{method}.

\begin{tthm}\label{lowerthm}
Consider ASEP with the initial data \eqref{IC} and $C=C(M)$ as in \eqref{C}.
We have
\begin{equation}\label{lower}
\lim_{M \to \infty} \lim_{t \to \infty}\Pb\left( x_{M+\lambda M^{1/3}}(t)\geq -\xi M^{1/3}\right)\geq F_{\GUE}(-\lambda)F_{\GUE}(\xi-\lambda)
\end{equation}
for $\lambda,\xi \in \R$.
\end{tthm}
\begin{proof}
We shall prove 
\begin{equation*}
\lim_{M \to \infty} \lim_{t \to \infty}\Pb\left( x_{M+\lambda M^{1/3}}(t)\geq -\xi M^{1/3}-M^{1/4}\right)\geq F_{\GUE}(-\lambda)F_{\GUE}(\xi-\lambda),
\end{equation*}
which is easily seen to imply \eqref{lower}.
Let (as in Proposition \ref{X0}) $\delta \in (0,1/2-7\nu /6)$  and $ \chi\in(\nu+\delta,1/2-\nu/6).$ 
Define the event
\begin{equation*}
\mathcal{E}_{t-t^{\chi}}^{\nu}=\{x_{0}(t-t^{\chi})\geq M+(\lambda-\xi)M^{1/3}\}\cap\{x_{M+\lambda M^{1/3}}(t-t^{\chi})\geq -t^{\delta+\nu}\},
\end{equation*}
for $\nu=0$, we get the event $ \mathcal{E}_{t-t^{\chi}}$ from \eqref{passauf}.
The application of the  Harris inequality  of Proposition \ref{FKG}
yields 
\begin{equation}\label{FKG2}
\Pb(\mathcal{E}_{t-t^{\chi}}^{\nu})\geq \Pb(\{x_{0}(t-t^{\chi})\geq M+(\lambda-\xi)M^{1/3}\})\Pb(\{x_{M+\lambda M^{1/3}}(t-t^{\chi})\geq -t^{\delta+\nu}\}).
\end{equation}
We will treat  each of the two factors on the R.H.S. of \eqref{FKG2} separately. 
By Proposition \ref{X0}, we obtain 
\begin{equation}
\lim_{M\to\infty}\lim_{t \to\infty}\Pb(\{x_{0}(t-t^{\chi})\geq M+(\lambda-\xi)M^{1/3}\})\geq F_{\GUE}(\xi-\lambda),
\end{equation}
(in fact,  this inequality is even  an identity).  As for the second factor,  using Theorem \ref{convthm}, we have 
\begin{equation}
\lim_{t\to\infty}\Pb(x_{M+\lambda M^{1/3}}^{A}(t-t^{\chi})\geq -t^{\delta+\nu})=F_{\GUE}(-\lambda)
\end{equation}
and the bound \eqref{EM} implies that 
\begin{equation}
\lim_{t\to\infty}\Pb(\{x_{M+\lambda M^{1/3}}^{A}(t-t^{\chi})\geq -t^{\delta+\nu}\}\cap\{x_{M+\lambda M^{1/3}}(t-t^{\chi})< -t^{\delta+\nu}\} )=0.
\end{equation}
Since $x_{M+\lambda M^{1/3}}^{A}(t-t^{\chi})\geq x_{M+\lambda M^{1/3}}(t-t^{\chi})$ under the basic coupling, we obtain  that 
\begin{equation}
\lim_{t\to\infty}\Pb(x_{M+\lambda M^{1/3}}(t-t^{\chi})\geq -t^{\delta+\nu})=F_{\GUE}(-\lambda).
\end{equation}
Thus in total we obtain 
\begin{equation*}
\lim_{M\to\infty}\lim_{t \to \infty}    \Pb(\mathcal{E}_{t-t^{\chi}}^{\nu})\geq F_{\GUE}(-\lambda)F_{\GUE}(\xi-\lambda).
\end{equation*}

The next step is to bound $x_{M+\lambda M^{1/3}}$ from below by a particle in a countable state space ASEP.
We start at time $t-t^{\chi}$ an ASEP from
\begin{equation}\label{eta-}
\tensor[^-]{\eta}{_{}}^{\nu}(j)=\mathbf{1}_{\{-\lfloor t^{\nu+\delta}\rfloor,\ldots,-\lfloor t^{\nu+\delta}\rfloor+M+\lambda M^{1/3}-1\}}(j)+\mathbf{1}_{\{j\geq M+(\lambda-\xi) M^{1/3}\}}(j), \quad j\in \Z,
\end{equation}
and   denote by $(\tensor[^-]{\eta}{_{\ell}}^{\nu})_{\ell\geq t-t^{\chi}}$ this ASEP. We denote  by $\tensor[^-]{x}{_{M+\lambda M^{1/3}}}(s)$  (suppressing the $\nu$) the position of the leftmost particle of 
$\tensor[^-]{\eta}{_{s}}^{\nu}$.
We have the relation, proven in a similar way as the relation \eqref{tildeE},   
\begin{equation*}\begin{aligned}
\mathcal{E}_{t-t^{\chi}}^{\nu} 
\subseteq\{\tensor[^-]{x}{_{M+\lambda M^{1/3}}}(s)\leq x_{M+\lambda M^{1/3}}(s),t-t^{\chi}\leq s \leq t\}.
\end{aligned}\end{equation*}
It is now essential that $t^{\chi}> t^{\nu+\delta}$: Because of this, the ASEP started from \eqref{eta-} has enough time to mix to equilibrium during $[t-t^{\chi},t]$ and 
hence $ \Pb( \tensor[^-]{x}{_{M+\lambda M^{1/3}}}(t)\geq -\xi M^{1/3}-M^{1/4})$ is almost one: Specifically, 
we  apply Proposition \ref{DOIT} with $R=M^{1/4}, \mathcal{M}=\lfloor t^{\nu+\delta}\rfloor+1,\varepsilon=1,$ and note that  $t^{\chi}> K\mathcal{M} $ for a constant $K$ and $t$ large enough.  We thus get 
\begin{align}\label{aligned}
 \Pb( \{\tensor[^-]{x}{_{M+\lambda M^{1/3}}}(t)\geq -\xi M^{1/3}-M^{1/4}\})\geq 1-C_{1}e^{-C_{2} M^{1/4}}-1/(\lfloor t^{\nu+\delta}\rfloor+1).
\end{align}

By construction,  $\tensor[^-]{x}{_{M+\lambda M^{1/3}}}(t)$  is independent of the event $\mathcal{E}_{t-t^{\chi}}^{\nu} .$  Thus we may conclude
\begin{equation*}
\begin{aligned}
&\lim_{M\to\infty}\lim_{t\to\infty}\Pb( x_{M+\lambda M^{1/3}}(t)\geq -\xi M^{1/3}-M^{1/4})
\\&\geq    \lim_{M\to\infty}\lim_{t\to\infty}\Pb( \{\tensor[^-]{x}{_{M+\lambda M^{1/3}}}(t)\geq -\xi M^{1/3}-M^{1/4}\}\cap \mathcal{E}_{t-t^{\chi}}^{\nu} )
\\&\geq   F_{\GUE}(-\lambda)F_{\GUE}(\xi-\lambda).
\end{aligned}
\end{equation*}

\end{proof}

\section{Proofs of   Theorems \ref{GUEGUE} and \ref{GUEGUE2}}\label{adapt}
While the proof of Theorem \ref{GUEGUE}  is immediate from the preceding results, the proof of Theorem \ref{GUEGUE2} requires some adaptions, which we give without repeating all the details given when proving Theorem \ref{GUEGUE}.
Let us start by proving Theorem \ref{GUEGUE}.
\begin{proof}[Proof of Theorem \ref{GUEGUE}]
By the inequality $y_{M+\lambda M^{1/3}}(t)\geq x_{M+\lambda M^{1/3}}(t) $ (see \eqref{Y}),  we see that Theorem \ref{GUEGUE} follows from  Corollary  \ref{GUEGUEcor} and Theorem  \ref{lowerthm}.
\end{proof}
Now we come to Theorem \ref{GUEGUE2}.
\begin{proof}[Proof of Theorem \ref{GUEGUE2}]
The structure of the proof of  Theorem \ref{GUEGUE2} is identical to the one for the  proof of Theorem \ref{GUEGUE}. To lighten the notation, we set
\begin{equation*}
M(t)=t^{\nu}+\lambda t^{\nu/3}.
\end{equation*}

 To prove \eqref{jaja}, we show separately the two inequalities 
\begin{align}\label{jajanee}
&\lim_{t \to \infty} \Pb\left(X_{M(t)}(t/(p-q))\geq -\xi t^{\nu/3}\right)\leq F_{\GUE}(-\lambda)F_{\GUE}(\xi-\lambda),
\\&\lim_{t \to \infty} \Pb\left(X_{M(t)}(t/(p-q))\geq -\xi t^{\nu/3}\right)\geq F_{\GUE}(-\lambda)F_{\GUE}(\xi-\lambda)\label{mart}.
\end{align}

 For \eqref{jajanee}, we define 
\begin{align}\nonumber
&X_{n}^{A}(0)=
-n-\lfloor t-2t^{(\nu+1)/2}\rfloor  \quad \mathrm{for} \quad   n \geq 1 
\\&X_{n}^{B}(0)=\label{XB}
-n  \quad \mathrm{for} \quad   n \geq -\lfloor t-2t^{(\nu+1)/2}\rfloor
\\&Y_{n}(t)=\min\{X_{n}^{A}(t),X_{n}^{B}(t)\}.
\end{align}
We have  $X_{n}(t)\leq Y_{n}(t)$ under the basic coupling  and thus prove \eqref{jajanee} by showing
\begin{align}\label{jajanee2}
&\lim_{t \to \infty} \Pb\left(Y_{M(t)}(t/(p-q))\geq -\xi t^{\nu/3}\right)= F_{\GUE}(-\lambda)F_{\GUE}(\xi-\lambda).
\end{align}
We can in fact prove  \eqref{jajanee2} for all $\nu\in (0,1)$.
We need as input the convergence 
\begin{equation}\label{doch}
\lim_{t\to\infty}\Pb\left(\frac{x_{M(t)}^{\mathrm{step}}(t/(p-q))-t+2t^{\nu/2+1/2}}{t^{1/2-\nu/6}}\geq -s\right)=F_{\GUE}(s-\lambda).
\end{equation}
As stated, \eqref{doch} does not seem to exist in the literature. However, Theorem 11.3 in \cite{BO17}  shows the convergence of the rescaled $x_{\sigma t}^{\mathrm{step}}$ for $\sigma$ bounded away from $0$ (see Remark 11.4 in \cite{BO17}). Inspecting the proof of 
Theorem 11.3 of \cite{BO17} reveals that the convergence to $F_{\GUE}$ follows from the convergence of the position of  rightmost particle of the continuous Laguerre orthogonal polynomial ensemble to $F_{\GUE}$, which also holds in the scaling of \eqref{doch}.

Analogous to  Proposition \ref{prop0} and proven in the same way we get
\begin{equation}\label{prop00}
  \lim_{t\to \infty}\Pb\left( \left| X_{M(t)}^{A}(t/(p-q))-X_{M(t)}^{A}((t-t^{\kappa})/(p-q))-t^{\kappa}+2t^{\kappa+\frac{\nu-1}{2}}\right|\geq \varepsilon t^{1/2-\nu/6}\right)=0
\end{equation}
 for $\kappa<1.$ 
 For $\kappa\in(1/2+\nu/2, 1)$ we can then  prove the analogue of Proposition \ref{indepprop}
 \begin{equation}\label{XX}
 \lim_{t\to \infty}\Pb(\min\{X_{M(t)}^{A}((t-t^{\kappa})/(p-q))+t^{\kappa}-2t^{\kappa+\frac{\nu-1}{2}},X_{M(t)}^{B}(t/(p-q))\}\geq -\xi t^{\nu/3})=F_{\GUE}(-\lambda)F_{\GUE}(\xi-\lambda).
 \end{equation}
 To have \eqref{XX}, we needed  to assume $\kappa>1/2+\nu/2$ so that for  $\varepsilon >0$  with $1/2+\nu/2+\varepsilon<\kappa$ we have on one hand that the leftmost hole of the initial data \eqref{XB}
 enters the space-time region 
 \begin{equation}\label{11}
 \{(i,s):i<-t^{1/2+\nu/2+\varepsilon},0\leq s\leq t/(p-q)\}
 \end{equation}
 with vanishing probability. On the other hand,  $X_{1}^{A}(s/(p-q)),0\leq s\leq t-t^{\kappa}$ enters the space-time region 
  \begin{equation}\label{22}
 \{(i,s):i>-t^{\kappa}/4,0\leq s\leq (t-t^{\nu})/(p-q)\}
 \end{equation}
 with vanishing probability. Since \eqref{11} and \eqref{22} are disjoint,  this shows the independence of $X_{M(t)}^{A}((t-t^{\kappa})/(p-q)), X_{M(t)}^{B}(t/(p-q))$ once they are restricted to \eqref{11},\eqref{22},  leading to \eqref{XX}.
 Finally, deducing \eqref{jajanee2} from \eqref{XX} is done exactly as in the proof of \eqref{prodform}.
 
 Next, to prove \eqref{mart}, we first prove the analogue of Proposition \ref{X0}, namely the convergence
 \begin{equation}\label{X02}
 \lim_{t\to\infty}\Pb(X_{0}((t-t^{\chi})/(p-q))\geq t^{\nu}-\xi t^{\nu/3})\geq F_{\GUE}(\xi),
 \end{equation}
with $\chi$ as in Proposition \ref{X0}. The proof of \eqref{X02} is analogous  to the one of Proposition \ref{X0}: one essentially  has to replace the term $M$ by $t^{\nu}$ in the proof of Proposition \ref{X0}, and instead of  the double limit we have a simple limit $t\to \infty.$ For example, the parameter $Z$ from \eqref{Z} now is
 \begin{equation}
 Z=t^{\nu}-\xi t^{\nu/3}+t^{\nu/4}
 \end{equation}
 and one checks that all steps of the proof go through with this choice. The same  applies to  the proof of \eqref{mart}, which  uses \eqref{X02} and is analogous to the proof of Theorem \ref{lowerthm}.
 
 Finally, the proof of \eqref{jaja2} (for $s>0,$ the case $s<0$ follows from \eqref{jajanee2}) is very similar to the one of \eqref{Thehound}, let us however explain how the restriction $\nu<3/7$ comes into play here as well : Similar to how  \eqref{Thehound} was proven, \eqref{jaja2} follows from
 \begin{equation}\label{EM3}
 \lim_{t\to\infty} \Pb(X_{M(t)}^{A}(t/(p-q))\geq -st^{1/2-\nu/6},X_{M(t)}(t/(p-q))< -st^{1/2-\nu/6})=0. 
 \end{equation}
Now one proves directly as \eqref{EM} that for $\delta>0$
 \begin{equation}
 \begin{aligned}\label{EM2}
& \Pb\left(\inf\{\ell:X_{M(t)}^{A}(\ell/(p-q))\neq X_{M(t)}(\ell/(p-q))\}\leq t/(p-q), X_{M(t)}(t/(p-q))\leq -(M(t)+1)t^{\delta/2}\right)
 \\&\leq (M(t)+1)C_{1}e^{-C_{2}t^{\delta/2}}\to_{t\to \infty} 0.
 \end{aligned}
 \end{equation}
 Now \eqref{EM2} implies \eqref{EM3} if $(M(t)+1)t^{\delta/2}<st^{1/2-\nu/6}$ which we can achieve if 
 $\nu<1/2-\nu/6,$ i.e. $\nu<3/7$.
 \end{proof}
  
\bibliographystyle{imsart-number}
\bibliography{Biblio}

\end{document}